\documentclass[a4paper,11pt]{article}
\usepackage{amsmath,mathtools}
\usepackage{amsmath}
\usepackage{amsfonts}
\usepackage{yfonts}
\usepackage{amssymb}
\usepackage{comment}
\usepackage{mathrsfs}
\usepackage{hyperref}
\usepackage[retainorgcmds]{IEEEtrantools}
\usepackage{enumerate}
\usepackage{enumitem}
\usepackage{amsthm}
\usepackage{fullpage}
\usepackage{graphicx}
\usepackage{caption}
\usepackage{color}
\usepackage{subfigure}
\usepackage{tikz}
\usepackage{tikz-qtree}
\usetikzlibrary{arrows,automata,positioning}
\usepackage{tabularx}

\theoremstyle{definition} \newtheorem{definition}{Definition}[section]
\theoremstyle{plain} \newtheorem{theorem}[definition]{Theorem}
\theoremstyle{plain} \newtheorem{corollary}[definition]{Corollary}
\theoremstyle{plain} \newtheorem{lemma}[definition]{Lemma}
\theoremstyle{definition} \newtheorem{remark}[definition]{Remark}
\theoremstyle{plain}
\newtheorem{proposition}[definition]{Propostion}
\theoremstyle{plain} 
\theoremstyle{plain} 
\theoremstyle{definition}
\newtheorem{example}[definition]{Example}
\usepackage{graphicx}
\usepackage[final]{microtype}

\theoremstyle{definition}

\newcommand{\im}{\mbox{im}}

\newcommand{\gen}[1]{\langle #1 \rangle}

\newcommand{\T}[1]{\mathcal{#1}}

\newcommand{\C}[2]{\mathfrak{#1}_{#2}}
\newcommand{\CCn}{\mathfrak{C}_{n}}
\newcommand{\CCnr}{\mathfrak{C}_{n,r}}

\newcommand{\xn}{X_{n}}
\newcommand{\xns}{X_{n}^{\ast}}
\newcommand{\xnp}{X_{n}^{+}}

\newcommand{\id}{\mbox{id}}
\newcommand {\core}[1]{\mbox{Core}(#1)}

\newcommand{\N}{\mathbb{N}}

\newcommand{\tr}[1]{\pi_{#1}}
\newcommand{\out}[1]{\lambda_{#1}}

\begin{document}
\title{ Conjugate subgroups and overgroups of $V_n$}
\author{C. Donoven, and F. Olukoya}
\maketitle

\begin{abstract}
We describe subgroups and overgroups of the generalised Thompson groups $V_n$ which arise via conjugation by rational homeomorphisms of Cantor space.  We specifically consider conjugating $V_n$ by homeomorphisms induced by synchronizing transducers and their inverses.   Our descriptions of the subgroups and overgroups use properties of the conjugating transducer to either restrict or augment the action of $V_n$ on Cantor space. 
\end{abstract}
\section{Introduction}
The Higman-Thompson groups $G_{n,r}$, introduced in \cite{GHigman}, are groups of homeomorphisms of a Cantor space $\mathfrak{C}_{n,r}$ . In \cite{BlkYMaisANav}, the automorphisms of $G_{n,r}$ are realised as conjugation by homeomorphisms of Cantor space. Specifically, \cite{BlkYMaisANav} shows that the subgroup of the rational group $\mathcal{R}_{n,r}$ \cite{GriNekSus} consisting of homeomorphisms induced by \emph{bi-synchronizing} transducers corresponds exactly to the automorphism group of $G_{n,r}$.  This paper explores the consequences of relaxing this bi-synchronizing condition.

 We study \emph{one-way synchronizing} transducers. It follows from \cite{BlkYMaisANav} that conjugation of the generalised Thompson groups $V_n = G_{n,1}$ by such transducers and their inverses, respectively,  produce subgroups and overgroups of $V_n$. Our main results describe these conjugate subgroups and overgroups.  

The generalised Thompson groups $V_n$ are groups of homeomorphisms of $n$-ary Cantor space that act on infinite words by exchanging finite prefixes. The subgroups we produce via conjugation restrict this action by limiting which finite prefixes can be exchanged. This restriction is determined by the inverse of the conjugating transducer. Conversely, the overgroups we produce via conjugation augment the action of $V_n$ on Cantor space.  Elements in the overgroups exchange finite prefixes, but are then allowed to act non-trivially on the infinite suffix.  The action on the suffixes is determined by groups of homeomorphisms of Cantor space associated to the conjugating transducer.

We begin by discussing the conjugate subgroups, before moving on to overgroups. We briefly introduce a few notions to state our main results. Note that we provide a thorough discussion of relevant definitions in Section~\ref{Section: preliminaries} and recommend \cite{GriNekSus} as a useful introduction to transducers. Let $X_n := \{ 0, 1, \ldots, n-1 \}$ be a finite set of symbols and let $\sim$ be an equivalence relation on $X_n^+$, the set of non-empty finite words over $X_n$.  Under certain conditions on $\sim$, one can consider the subgroup of $V_n$ that only exchanges prefixes related by $\sim$. We call this subgroup $V_{\sim}$.   

One method for producing such equivalence relations is to use automata. (Note that our use of the words transducer and automaton is specific. In essence, an automaton is a `machine' that reads words, and a transducer is an automaton that outputs words.) Let $A=\langle X_n, Q, \pi, q_0\rangle$ be a finite automaton with start state $q_0$. One can define $\sim_A$ to be the equivalence relation on $X_n^+$ such that two words are related if and only if when they are read by $A$, the active state of the automaton is the same. Although this relation does not necessarily give rise to a subgroup of $V_{n}$, we can `collapse' $A$ to get an automaton that does produce such a relation. We discuss this process in Section~\ref{Sec: subgroups}.     

 Our main result describing the conjugate subgroups of $V_n$ is the following theorem.
 
 \begin{theorem}\label{thm: introsubgroups}
 Let $T_{q_0} = \gen{X_n, Q, \pi, \lambda}$ be an initial one-way synchronizing transducer with inverse $S_{(\epsilon, q_0)} = \gen{X_n, Q', \pi', \lambda'}$. Let $A = \gen{X_n, Q', \pi', q_0}$ be the automaton obtained from $S$ by forgetting outputs. Then  $V_n^{T_{q_0}}  = V_{\sim_{A}}$.
 \end{theorem} 
 
Theorem~\ref{thm: introsubgroups}, demonstrates that conjugates of $V_n$ by one-way synchronizing transducers are subgroups.  Conversely, conjugating $V_n$ by the inverse of a one-way synchronizing transducer produces overgroups.  These overgroups, in essence, combine the action of $V_n$ with other groups of homeomorphisms of Cantor space.  We describe these augmenting groups next, beginning with the synchronous case. 
 
Given an invertible, synchronous transducer $T$, the initial transducer $T_q$ induces a homeomorphism of $\CCn$ for each state $q$ of $T$. We can consider the subgroup of the rational group generated by these homeomorphisms, denoted by $\T{G}(T)$. See \cite{Nek_selfsimilargroups} for more details on these groups, which belong to a class called self-similar groups. By adding extra generators to $V_n$ that act as elements of $\T{G}(T)$, we construct the overgroup $V_n({\T{G}(T)})$, whose definition is in Section~\ref{Section: overgroups}.  These groups, which combine the actions of $V_n$ and self-similar groups, are known as Nekrashevych groups and were studied in \cite{Nek04}.  A famous example built using the Grigorchuk group was examined by R\"{o}ver in \cite{Rover_99}. 

If $T$ is an asynchronous transducer, however, then there exists states of $T$ that do not induce homeomorphisms of $\CCn$.  In Section~\ref{Section: overgroups}, we generalise the construction of $V_n(\T{G}(T))$ to treat this case. Specifically, we associate to an arbitrary transducer $T$, satisfying certain other conditions, a group $\T{H}(T)$ of rational homeomorphisms of $\CCn$.  It is important to note that when $T$ is invertible with inverse $S$, $\T{H}(TS)$ is well defined and $V_n\leq \T{H}(TS)$. Also, if $T$ is synchronous, then the group $\T{H}(T) = V_n({\T{G}(T)})$. 
 
Our results characterising the conjugate overgroups $V_n$ are presented in the following theorem.

\begin{theorem} 
Let $T_{q_0}= \gen{X_n, Q_{T}, \pi_{T}, \lambda_{T}}$ be an initial one-way synchronizing transducer with inverse $S_{(\epsilon, q_0)}$. Then $V_n^{{T^{-1}_{q_0}}} = \T{H}(TS)$. If $T$ is synchronous, we have $V_n^{{T^{-1}_{q_0}}} = V_n({\T{G}(TS)})$.
\end{theorem}

If $T$ is bi-synchronizing and invertible, then $\T{H}(TS) = V_n$ and, as the relation $\sim_{A}$ is trivial, $V_{\sim_{A}}=V_n$ also.

In Section~\ref{Section: Synch conditions}, we also provide an interesting characterisation of the group $\T{G}(TS)$, when $T$ is synchronous, invertible, and one-way synchronizing with inverse $S$.

\begin{proposition}\label{Proposition: Intro}
Let $T = \gen{X_n, Q, \pi, \lambda}$ be a finite minimal, one-way synchronizing, synchronous transducer and let $S = \gen{X_n,Q^{-1}, \pi', \lambda'}$ be its inverse. Also let $0 < k \in \mathbb{N}$ be the minimal synchronizing length of $T$. The automaton group generated by $TS$ is isomorphic the subgroup of the symmetric group $S_{n^k}$ generated by the permutations $\lambda_{TS}(\centerdot, p): X_n^k \to X_n^k$, where $p$ a state of $TS$ and $\lambda_{TS}$ is the transition function for $TS$. In particular, this group is finite.
\end{proposition}

The outline of the paper is as follows. In Section~\ref{Section: preliminaries}, we set up our notation, introduce transducers, and provide a framework for our results. In Section~\ref{Section: Synch conditions}, we prove some preliminary results including Proposition~\ref{Proposition: Intro}. Section~\ref{Sec: subgroups} describes the subgroups resulting from Proposition~\ref{pro:conjproducessubgroups} and Section~\ref{Section: overgroups} describes the overgroups from Corollary~\ref{conjproducesovergroups}. We finish our work in Section~\ref{Section: H(T)}, investigating the simplicity of the groups $\T{H}(T)$, specifically giving conditions on when their derived subgroups are simple.  The article concludes in Section~\ref{Section: Questions} with some open problems. 
\newline

{\bf{Acknowledgements.}} The authors  would like to thank Collin Bleak and Matthew Brin for their support and feedback on early versions of the paper.  The second author also acknowledges the financial support of The Carnegie Trust for the Universities of Scotland.
\section{Preliminaries}\label{Section: preliminaries}
We begin with a brief description of transducers and their properties.  For more detail, see \cite{GriNekSus}, from which we draw notation and the inversion algorithm.

Throughout this paper, we use the alphabet $X_n=\{0,1,\ldots, n-1\}$, so that Cantor space can be expressed as $\CCn=X_n^{\omega}$. We shall denote by $X_n^{*}$ the set of all finite strings (including the empty string) in the alphabet $X_n$. The empty string we shall denote $\epsilon$. For a word $\Gamma \in X_n^{*}$ we shall denote by $|\Gamma|$ the length of the word $\Gamma$; for the empty string $\epsilon$ we shall set $|\epsilon| = 0$. Set $X_n^{+} := X_n^{*} \backslash \{\epsilon\}$ and for $k \in \mathbb{N}$ we shall denote by $X_n^{k}$ the subset of $X_n^{*}$ consisting of all words of length $k$.  For a word $\alpha\in X_n^*$, the set $[\alpha]=\{\alpha\chi\in \CCn\;|\;\chi\in \CCn\}$ is a called a \emph{cone} and the set of all cones forms a basis for the topology on Cantor space.  Set $H(\CCn)$ to be the group of homeomorphisms of Cantor space. For $h, g \in H(\CCn)$ we shall write $h^{g}$ for $g^{-1} h g$. Unless otherwise indicated we shall write functions to the right of their arguments.

For two words $u,v\in X_n^{*}$ such that $u$ is a prefix of $v$, we write$v-u$ for the word $\mu \in X_{n}^{*}$ such that $v = u \mu$. Given two sets $V \subset X_n^{*}$ and $W \subset X_n^{*} \sqcup X_n^{\omega}$ we shall denote by $VW$ the set $\{vw \mid v \in V, w \in W \}$. If $V$ is a singleton, $V= \{ \rho\}$ then we shall write $\rho W$ for the set $VW$.

A \emph{transducer} $T=\langle X_n, Q, \pi, \lambda\rangle$ is a tuple consisting of an alphabet $X_n$, a finite set of states $Q$, and two functions $\pi:X_n\times Q\to Q$ and $\lambda:X_n \times Q \to X_n^*$.  Let $q_0\in Q$. An \emph{initial transducer} $T_{q_0}$ is a transducer with active state $q_0$.  Let $x_1x_2x_3\ldots$ be a finite or infinite word over $X_n$. Initial transducers are recursively defined to operate on words (finite and infinite) according to the following rule:
\[(x_1x_2x_3\ldots)T_{q_0}=\lambda(x_1,q_0)(x_2x_3\ldots)T_{\pi(x_1,q_0)}\]
where, after rewriting $x_1$ by $\lambda(x_1,q_0)$, $T$ transitions to state $\pi(x_1,q_0)$ and continues processing letters.  Often, we extend $\lambda$ and $\pi$ to accept finite words, rather than just single letters, in the natural way. Note that all transducers considered in this paper are finite. 

For a state $q$ of $Q$, we call the set $\{ (\alpha)T_{q} \mid \alpha \in \CCn \} \subseteq \CCn$ the \emph{image of $q$} and we denote it $\im(q)$. If $T_{q}$ is injective, then we say that \emph{$q$ is an injective state}; if $T_{q}$ is a homeomorphism, then we say that \emph{$q$ is a homeomorphism state}, otherwise we say that \emph{$q$ is a non-homeomorphism state}. 

For two states $q_1$ and $q_2$ of a transducer $T = \gen{X_n, Q, \pi, \lambda}$ we say that \emph{$q_2$ is accessible from $q_1$} if there is a word $w \in X_n^{*}$ such that $\pi(w, q_1) = q_2$. If every state of $T$ is accessible from every other then $T$ is called \emph{accessible}.

A transducer $T$ is \emph{synchronous} if for all states $q$ of $T$ the map $\lambda(\centerdot, q)$ is a map from $X_n$ to itself i.e. $T$ preserves the length of words, and is \emph{asynchronous} otherwise. A state $q\in Q$ of an asynchronous transducer is called a state of \emph{incomplete response} if, for a certain $i\in X_n$, the longest common prefix of the words $\lambda(i\chi,q)$ for all $\chi\in \CCn$ is different than $\lambda(i,q)$. We say that two initial transducers $T_t$ and $R_r$ are \emph{$\omega-$equivalent} if they define the same function on $\CCn$. For an initial transducer $T_{t_0}$ and states $t_1$ and $t_2$ of $T_{t_0}$ we say that the states $t_1$ and $t_2$ are $\omega$-equivalent if  the initial transducers $T_{t_1}$ and $T_{t_2}$ are $\omega$-equivalent. An initial transducer $T_{t_0}$ with no states of incomplete response and no pair of $\omega$-equivalent states is called a \emph{minimal transducer}. For each initial transducer  $T_{t}$, there is a unique minimal transducer $\omega-$equivalent to $T_{t}$ (see \cite{GriNekSus}). 

For a non-initial transducer $T$, we say $T$ is \emph{ $\omega$-minimal} if it has no pair of $\omega$-equivalent states. For  a state $t$ of $T$ we shall let $\{t\}_{\omega}$ denote the $\omega$-equivalence class of $t$ in the transducer $T$. For a given non-initial transducer $T$ there is an $\omega$-minimal transducer $R$ such that for a state $t$ of $T$ and any state $t'$ in $\{t\}_{\omega}$ there is a unique state $r$ of $R$ such that the initial transducers $T_{t'}$ and $R_{r}$ are equivalent. We call $R$ an \emph{$\omega$-minimal transducer representing $T$}. Two $\omega$-minimal transducers $T = \gen{X_n, Q_T, \pi_T, \lambda_T}$ and $R = \gen{X_n, Q_R, \pi_R, \lambda_R}$  are called \emph{isomorphic} if  there is a bijection $g: Q_T \to Q_R$ such that for a state $t \in Q_T$ the initial transducers $T_{t}$ and $R_{(t)g}$ are $\omega$-equivalent. Two non-initial transducers $T$ and $R$ are called \emph{equivalent} if the $\omega$-minimal transducer representing $T$ is isomorphic to the $\omega$-minimal transducer representing $R$. We denote by $\{T\}_{\omega}$ the equivalence class of the non-initial transducer $T$. For a non-initial transducer $T$, there is a unique $\omega$-minimal transducer up-to isomorphism representing $T$.

If an initial transducer $T_{q_0}$ defines an invertible function from $\CCn$ to itself, the function is in fact a homeomorphism and we will often identify the initial transducer with the homeomorphism it produces.  Homeomorphisms induced by transducers are called \emph{rational}.  The inverse of the homeomorphism $T_{q_0}$ is also rational and an initial transducer that produces the inverse is given by the following algorithm. 

We first define the function $\T{L}_q:X_n^*\to X_n^*$ for $q\in Q$, where $\T{L}_q(\alpha)$ is the longest common prefix of all words in the preimage of $[\alpha]$ under $T_q$.  Note that if $T_q$ is a homeomorphism, $\T{L}_q(\epsilon)=\epsilon$. 

\begin{proposition}[Inversion Algorithm \cite{GriNekSus}]\label{Proposition:inversionalgorithm}
Let $T=\langle X_n, Q, \pi, \lambda\rangle$ be a transducer such that $T_{q_0}$ is a homeomorphism. Also let, for each state $q\in Q$, $\{\alpha_1,\alpha_2,\ldots,\alpha_{m_q}\}\subseteq X_n^*$ be a set of finite words such that $\T{L}_q(\alpha_i)=\epsilon$.  Let \[Q'=\cup_q\{(\alpha_1,q),(\alpha_2,q)\ldots,(\alpha_{m_q},q)\}.\]For arbitrary $(\alpha_i,q)\in Q'$ and $\beta\in X_n^*$, we set \[\lambda'(\beta,(\alpha_i,q))=\T{L}_q(\alpha_i\beta),\]\[\pi'(\beta,(\alpha_i,q))=\big(\alpha_i\beta-\lambda(\T{L}_q(\alpha_i\beta),q),\pi(\T{L}_q(\alpha_i\beta),q)\big).\] Then, the transducer $S=\langle X_n, Q', \pi', \lambda'\rangle$ is a transducer such that $S_{(\epsilon,q_0)}$ is the inverse homeomorphism of $T_{q_0}$.
\end{proposition}

We will often let $T$ be an invertible transducer and $S$ be its inverse. By this, we mean that $T$ has a state $q_0$ such that $T_{q_0}$ is a homeomorphism and every state of $T$ is accessible from $q_0$ and $S$ has a state $(\epsilon, q_0)$ such that $S_{(\epsilon,q_0)}$ produces the inverse homeomorphism of $T_{q_0}$ and every state of $S$ is accessible from $(\epsilon, q_0)$. We observe that every pair $(w,q)$ which is accessible from $(\epsilon, q_0)$ in addition to having $\T{L}_{q}(w) = \epsilon$, satisfy $[w] \subset  \im(q)$. This follows by an easy induction argument, observing that $\CCn=[\epsilon] \subset \im(q_0)$. This means that the states of $S_{(\epsilon, q_0)}$ are, in some cases, a proper subset of $Q'$ as defined in Proposition~\ref{Proposition:inversionalgorithm}. In order to avoid having cluttered and confusing notion, throughout the paper we will generally use $S$ as the inverse transducer for $T$, so that $T_{q_0}^{-1}$ can represent the inverse of the homeomorphism associated to $T_{q_0}$.

If $T$ is an invertible synchronous transducer, the inverse of $T$ is much simpler to compute. The inverse of $T$ is the transducer $S = \gen{X_n,Q^{-1}, \pi', \lambda'}$, where ${}^{-1}:Q\to Q^{-1}$ is a bijection and
\begin{IEEEeqnarray*}{rCl}
\pi(\alpha,q) &=& p  \qquad \iff \qquad  \pi'(\beta,q^{-1}) = p^{-1} \nonumber \\
\lambda(\alpha,q) &=& \beta \qquad \iff \qquad \lambda'(\beta,q^{-1}) = \alpha
\end{IEEEeqnarray*} 
for appropriate $\alpha,\beta\in \CCn$. In particular the map $\lambda'(\centerdot,q^{-1}): X_n \to X_n$ is the inverse of $\lambda(\centerdot, q)$. For a state $q_0$ of $T$ we shall write $S_{q_0^{-1}}$ for the transducer representing the homeomorphism $T_{q_0}^{-1}$.

The product of transducers $T=\langle X_n, Q_T, \pi_T, \lambda_T\rangle$ and $R=\langle X_n, Q_R, \pi_R, \lambda_R\rangle$ is the transducer $TR=\langle X_n, Q_T\times Q_R, \pi_{TR}, \lambda_{TR}\rangle$, where $\pi_{TR}(i,(t,r))=\pi_R(\lambda_T(i,t),\pi_T(i,t))$ and $\lambda_{TR}(i,(t,r))=\lambda_R(\lambda_T(i,t),\pi_T(i,t))$ when $i\in X_n$, $t\in Q_T$, and $r\in Q_R$. 

We have a few observations:
\begin{remark}\label{Remark: some comments about product}\leavevmode
\begin{enumerate}[label = (\arabic*)]
\item Observe that for $\delta \in \CCn$ and a state $(t,r)$ of the product $TR$ we have that $\lambda_{TR}(\delta,(t,r) ) = ((\delta)T_t)T_{r}$. \label{Remark: comments about product point 1}
\item If $T_1, T_2, R_1, R_2$ are transducers such that $T_1 \in \{T_2\}_{\omega}$ and $R_1 \in \{R_2\}_{\omega}$. Then  $T_1 R_1$ is equivalent to $T_2 R_2$. This is essentially a consequence of point \ref{Remark: comments about product point 1}. Moreover for three non-initial transducers: $T_1$, $T_2$ and $T_3$  we have $(T_1 T_2)T_3$ is equivalent to $T_1(T_2T_3)$. This is because for three states $t_1 \in Q_{t_1}$, $t_2 \in Q_{t_2}$ and $t_3 \in Q_{t_3}$ the states $((t_1, t_2), t_3)$ and $(t_1, (t_2, t_3))$ represent the composition of the functions $(T_1)_{t_1}$, $(T_2)_{t_2}$ and $(T_3)_{t_3}$. \label{Remark: transducer product is associative on equivalent transducers and respects equivalence classes}
\item Note that if $S_{(\epsilon,q_0)}$ is the inverse of $T_{q_0}$, then $ST$ and $TS$ are not necessarily trivial transducers, but the states $((\epsilon,q_0),q_0)$ and $(q_0,(\epsilon,q_0))$ do correspond to the identity homeomorphism.
 
\end{enumerate}
 
\end{remark}

A transducer $T=\langle X_n, Q_T, \pi_T, \lambda_T\rangle$ is \emph{synchronizing at level $k$} if there is a function $f:X_n^k\to Q$ such that $\pi(\alpha,q)=f(\alpha)$ for each $\alpha\in X_n^k$ and $q\in Q$.  This means that after reading long enough words, the active state of the transducer does not depend on the initial state. This is stronger notion of synchronization than occurs elsewhere in the literature as in \cite{volkov2008synchronizing} for instance. A transducer $T$ is \emph{synchronizing} if such a $k$ exists and is \emph{bi-synchronizing} if both $T$ and its inverse are \emph{synchronizing}.  In this paper, we place particular importance on transducers that are synchronizing but whose inverses are not.  We call this \emph{one-way synchronizing}.  The \emph{core} of a transducer that is synchronizing at level $k$ is the set of states $C\subseteq Q$ one can reach after reading words of length $k$, i.e. $C=\{\pi(\alpha,q)|\alpha\in X_n^k, q\in Q\}$.  The core is necessarily a strongly connected component of $T$ (viewing $T$ as a graph with states $Q$ and edges corresponding to $\pi$), and we denote it by $\core{T}$.

The generalised Thompson groups $V_n$ can be described in several ways.  See \cite{GHigman} for more details.  In this paper, we use a characterisation of $V_n$ in terms of synchronizing transducers.  We define $V_n$ to be the group of rational homeomorphisms of $\CCn$ realised by bi-synchronizing transducers with trivial core, meaning that the core consists of one state, $id$ for which $\lambda(i,id)=i$ for each $i\in X_n$. This implies that after perhaps changing finite prefixes, every element of $V_n$ eventually acts trivially on infinite words. 

A \emph{small swap} is an element $(\alpha, \beta)\in V_n$, where $\alpha,\beta\in X_n^+$, that maps the finite prefix $\alpha$ to $\beta$ and vice versa.  Specifically, $(\alpha\chi)(\alpha, \beta)= \beta\chi$ and $(\beta\chi)(\alpha, \beta)= \alpha\chi$ for all $\chi\in X_n^\omega$ and fixes all other infinite words. Note that $\alpha$ must be incomparable to $\beta$ a small swap to be well defined. The set of all small swaps generates $V_n$.  See \cite{CBleakMQuick} for more details. 

The following proposition and corollary give the framework for our results, describing the result of conjugating $V_n$ by a one-way synchronizing transducer and its inverse repsectively. The proofs of Proposition \ref{pro:conjproducessubgroups} and Corollary \ref{conjproducesovergroups} are direct results from \cite{BlkYMaisANav}.  

\begin{proposition} \label{pro:conjproducessubgroups}
Let $T_{q_0} = \gen{X_n, Q, \pi, \lambda}$ be an initial synchronizing, transducer representing a homeomorphism of $\C{C}{n}$, then $T_{q_0}^{-1}V_nT_{q_0}$ is a subset of $V_n$, and so a subgroup.
\end{proposition}
\begin{proof}
Let $0 < k \in \mathbb{N}$ be a number such that $T_{q_0}$ is synchronizing at level $k$. Let $f \in  V_n$, we shall represent $f$ by a bi-synchronizing, initial transducer $B_f$ with trivial core. Let $l$ be the minimum synchronizing level of $B_f$. We may assume that $k >l$ (since if a transducer is synchronizing at some level, then it is synchronizing at every level beyond that).

Now since $T_{q_0}$ is finite and induces a self-homeomorphism of $\C{C}{n}$, there is a natural number $M$ such that for all words $\alpha \in X_n^M$,  $(\alpha)T_{q_0}$ has size at least $2k$. We may assume that $M > k$.

Let $T^{-1}_{q'_0} = \gen{X_n, Q', \pi',\lambda'}$ be the minimal transducer representing $T^{-1}_{q_0}$ and let $\alpha \in X_n^{M}$. Let $q' \in Q'$ be  such that $q' = \pi'(\alpha, q'_0 )$. Consider $\beta \gamma = (\alpha)T^{-1}_{q'_0}$, a word of length greater than or equal to $2k$, and assume that $\beta$ is the length $k$ prefix,so that $|\gamma| >k$.

Since $B_f$ is synchronizing at level $k$, and has trivial core, then it must be the case that $(\beta\gamma)B_f = (\beta) B_f \gamma$.  Let $p$ be the state of $T_{q_0}$ such that $\pi(\Lambda, s) = p$ for all states $s \in Q$. Hence after reading $\alpha$ through the product $T^{-1}_{q'_0}B_fT_{q_0}$, the active state is the triple $(q', \id, p)$. 

On the other hand after reading $\alpha$ through the product $T^{-1}_{q'_0}T_{q_0}$ the active state is $(q',p)$. However since $T^{-1}_{q'_0}T_{q_0}$ is just the identity map,  then $(q',p)$ is $\omega$-equivalent to the identity map, after removing states of incomplete response. 
\end{proof}

As mentioned above, an immediate consequence of the above proposition is:

\begin{corollary}\label{conjproducesovergroups}
Let $T_{q_0} = \gen{X_n, Q, \pi, \lambda}$ be an initial one-way synchronizing, transducer representing a homeomorphism of $\C{C}{n}$, then $T_{q_0}V_nT_{q_0}^{-1}$ contains $V_n$, and so is an overgroup of $V_n$.
\end{corollary}

Note that each of the subgroups and overgroups created as in Proposition \ref{pro:conjproducessubgroups} and Corollary \ref{conjproducesovergroups} are isomorphic to $V_n$. An easy inductive argument shows that the following  chain of inclusions are valid for a one-way synchronizing transducer $T_{q_0}$ as in the statement of the proposition: 
\[ \ldots  V^{T_{q_0}^{3}} \subsetneq V^{T_{q_0}^{2}} \subsetneq V^{T_{q_0}} \subsetneq V \subsetneq V^{T_{q_0}^{-1}} \subsetneq V^{T_{q_0}^{-2}} \subsetneq V^{T_{q_0}^{-3}} \ldots \]

It is also interesting to note these concepts also apply in the setting of the generalised Thompson groups $F_n$ and $T_n$, under additional assumptions.  For example, $F_n^{T_{q_0}}$ is a subgroup of $F_n$ when $T_{q_0}$ is synchronizing and preserves the lexicographical order on $\CCn$.

\section{Conditions for bi-synchronicity and the automaton group generated by \texorpdfstring{$TS$}{Lg}} \label{Section: Synch conditions}
In this section, we present some useful characterisations of bi-synchronicity and the automaton group generated by the product of a transducer and its inverse.  Our results focus on synchronous transducers, however we mention connections to asynchronous transducers. 

The following proposition is a characterisation of when a finite, synchronous transducer is one way synchronizing. First we observe that if  $T = \gen{X_n, Q, \pi, \lambda}$ is an invertible synchronous transducer, then the function $\lambda(\centerdot, q): X_n \to X_n$ for all states $q \in Q$ defines a permutation.

\begin{proposition} \label{conditionsonsynch}
Let $T = \gen{X_n, Q, \pi, \lambda}$ be a minimal (under $\omega$-equivalence) synchronous transducer and let $S = \gen{X_n,Q^{-1}, \pi', \lambda'}$ be its inverse. Then $T$ is synchronizing with minimal synchronizing level $k \in \mathbb{N}$ if and only if for any  fixed $p^{-1} \in Q^{-1}$ the product transducer $S_{p^{-1}}T_{q}$ is an element of $V_n$ for any  $q \in Q$ and $k$ is minimal such that for any word $\alpha \in X_n^k$, for any $\chi \in \C{C}{n}{}$, and for any $q \in Q$,  $(\alpha \chi)S_{p^{-1}}T_{q} = (\alpha)S_{p^{-1}}T_{q}\chi$.
\end{proposition}
\begin{proof}
The forward implication is a well-known result and is seen in \cite{SilvaSteinberg} for instance.

For the reverse implication, let $p^{-1} \in Q^{-1}$ be as in the statement of the proposition. Since $S_{p^{-1}}T_{q} \in V_n$ for all $q \in Q$, and $|Q| < \infty$, let $k$ be minimal so that for all $\alpha \in X_n^k$ and $\chi \in \C{C}{n}{}$, we have $(\alpha \chi)S_{p^{-1}}T_{q} = (\alpha)S_{p^{-1}}T_{q} \chi$. In other words after processing a word of length $k$ through $S_{p^{-1}}T_{q}$ for any $q \in Q$, we enter the identity state.

Fix $\alpha \in X_n^{k}$. Let $\beta$ be  the word of length $k$ such that $\lambda'(\beta, p^{-1}) = \alpha$ and suppose that  $\pi'(\beta, p^{-1}) = r^{-1}$.

Let $q \in Q$ be arbitrary. By the condition that after processing a word of length $k$ through $S_{p^{-1}}T_{q}$ for any $q \in Q$, we enter the identity state, and by minimality of $T$, we must have that $\pi(\alpha,q) = r$. Therefore $\pi(\alpha, \centerdot) : Q \to Q$ is a single valued function, and since $\alpha$ was arbitrary, this demonstrates that $T$ is synchronizing at level $k$.

In order to show that $k$ is the minimal synchronizing level of $T$ we argue that $T$ is not synchronizing at level $k-1$. This straight-forward, since by minimality of $k$, there exists a $\gamma \in X_n^{k-1}$ and a $q \in Q$ such that, if $\delta := \lambda'(\gamma,p^{-1})$, $t:= \pi(\delta,q)$ and $s^{-1}:= \pi'(\gamma,p^{-1})$, then $S_{s}T_{t}$ is not the identity state. This then implies that $\pi(\delta,p) = s$ (by definition of the inverse transducer) and $s$  is not $\omega$-equivalent to  $t$ (by minimality of $T$ and the definition of the inverse transducer). That is the function $\pi(\delta,\centerdot):Q \to Q$ takes at least two values.

\end{proof}

We remark that the above does not really make much sense in the context of asynchronous transducers since a transducer which is asynchronous must necessarily have an injective but not surjective state. (It was proved in \cite{BlkYMaisANav} that if all states of such a transducer induced self-homeomorphisms of $\C{C}{n}{}$ then the transducer is synchronous.) However, for those states of an asynchronous transducer which are homeomorphisms, the forward implication of the Proposition \ref{conditionsonsynch} holds. That is $S_{p^{-1}}T_{q}$ is an element of $V_n$, whenever $T$ is synchronizing, and $p^{-1}$ and $q$ are states of $S$ and $T$ respectively which induce self-homeomorphisms of $\C{C}{n}{}$. The proof is analogous to the proof of Proposition \ref{pro:conjproducessubgroups}.

 The next result concerns the automaton group, $\T{G}(TS)$ generated by $TS$, that is the subgroup of the group of homeomorphisms of $\C{C}{n}{}$ generated by $TS_{p}$ for all states $p$ of the transducer $TS$.  In the asynchronous case, we can take this to be the group generated by $S_{p^{-1}}T_{q}$ where $p^{-1}$ and $q$ are states of $S$ and $T$ respectively which induce homeomorphisms of $\C{C}{n}{}$. 
 
 Let $T$ be a synchronous, invertible transducer with inverse $S$. Fix a state $t$ of $T$ and consider the element \[T_{a_1}S_{b_1^{-1}}\ldots T_{a_m}S_{b_m^{-1}}\in \T{G}(TS),\] where $a_i, 1 \le i \le m$ are states of $T$ and $b_i^{-1}, 1 \le i \le m$ are states of $S$. By Proposition \ref{conditionsonsynch} we have:
 \[
   S_{t^{-1}}(T_{a_1}S_{b_1^{-1}}T_{a_2}T^{-1}_{b_2^{-1}}\ldots T_{a_m}S_{b_m^{-1}})T_{t} \in \T{G}(ST) \le V_n
 \]
Therefore conjugation by  $T_{t}$ induces an isomorphism from $\T{G}(TS)$ to a subgroup of $\T{G}(ST)$.  Analogously conjugation by $S_{t^{-1}}$ induces an isomorphism from $\T{G}(ST)$ to a subgroup of $\T{G}(TS)$. Hence the cardinality of $\T{G}(TS)$ is equal to the cardinality of $\T{G}(ST)$. If both are have finite cardinality, then it follows that $\T{G}(ST) \cong \T{G}(TS)$. Notice that, essentially as a corollary of Proposition \ref{conditionsonsynch}, the group $\T{G}(ST)$ is a subgroup of $V_n$.

In the case where $T$ is a synchronizing transducer, we are able to say a lot more about what $\T{G}(TS)$ looks like. This forms the content of the proposition below. 

\begin{proposition}
Let $T = \gen{X_n, Q, \pi, \lambda}$ be a minimal, synchronizing, synchronous transducer and let $S = \gen{X_n,Q^{-1}, \pi', \lambda'}$ be its inverse. Also let $0 < k \in \mathbb{N}$ be the minimal synchronizing length of $T$. The automaton group generated by $TS$ is isomorphic the subgroup of the symmetric group on $n^k$ points generated by the permutations $\lambda_{TS}(\centerdot, p): X_n^k \to X_n^k$, where $p$ a state of $TS$ and $\lambda_{TS}$ is the transition function for $TS$. In particular, this group is finite.
\end{proposition}
\begin{proof}
Set $TS := \gen{X_n, QQ^{-1}, \pi_{TS}, \lambda_{TS}}$ and let $k$ be as in the statement of the proposition. Also let $p,q \in Q$, and word $\alpha \in X_n^{k}$. Then if $\pi'(\alpha, p^{-1}) = r^{-1}$ and $\beta = \lambda'(\alpha,p^{-1})$, we must have that, $\pi(\beta, p) = r$. Since $T$ is synchronizing, we must also have that $\pi(\beta,q) = r$. Therefore the state $(p^{-1},q)$ of $ST$ acts on $\C{C}{n}{}$, by acting as  the permutation of words of length $k$ induced by $(p^{-1},q)$ on the length $k$ prefix, and as the identity on the infinite suffix preceding this.

For a state $p$ of $T$ let $\bar{p}$ denote the permutation induced by the action of $p$ on words of length $k$.  Likewise for a state $q^{-1}$ of $S$, the permutation $\bar{q}^{-1}$ corresponds to the induced permutation on words of length $k$ by $S$. Define a map from $\T{G}(TS) \to S_{n^k}$ (the symmetric group on $n^k$ letters), by 
\[
T_{a_1}S_{b_1^{-1}}T_{a_2}S_{b_2^{-1}}\ldots T_{a_m}S_{b_m^{-1}} \mapsto \bar{a}_1\bar{b}_1^{-1}\bar{a}_2\bar{b}_{2}^{-1}\ldots\bar{a}_m\bar{b}_m^{-1}
\]
where $a_i, \ 1 \le i \le m$ are states of $T$ and $b_i^{-1}, 1 \le i \le m$ are states of $S$. 

This map is clearly a homomorphism. In order to show that it is a monomorphism, we must argue that its kernel is trivial.

Let $T_{a_1}S_{b_1^{-1}}T_{a_2}S_{b_2^{-1}}\ldots T_{a_m}S_{b_m^{-1}} \in \T{G}(TS)$ be such that $\bar{a}_1\bar{b}_1^{-1}\bar{a}_2\bar{b}_{2}^{-1}\ldots\bar{a}_m\bar{b}_m^{-1}  = id \in S_{n^k}$. Rearranging, we have $\bar{b}_1\bar{a}_2\bar{b}_{2}\ldots\bar{a}_m = \bar{a}_1^{-1}\bar{b}_m$. Now notice that since $S_{b_1^{-1}}T_{a_2}S_{b_2^{-1}}\ldots T_{a_m} \in \T{G}(ST)$  and by the observation in the first paragraph, it acts on $\alpha\chi \in \C{C}{n}{}$ by,
\[
 \alpha\chi \mapsto (\alpha)\bar{b}_1^{-1}\bar{a}_2\bar{b}_{2}^{-1}\ldots\bar{a}_m \chi = (\alpha)\bar{a}_1^{-1}\bar{b}_m\chi = (\alpha \chi) S_{a_{1}^{-1}}T_{b_{m}}
\]

Therefore $S_{b_1^{-1}}T_{a_2}S_{b_2^{-1}}\ldots T_{a_m}$ is $\omega$-equivalent to $S_{a_{1}^{-1}}T_{b_{m}}$. This means that $T_{a_1}S_{b_1^{-1}}\ldots T_{a_m}S_{b_m^{-1}}$ is $\omega$-equivalent to $id \in \mathrm{Homeo}(\C{C}{n}{})$.
\end{proof}

A consequence of the proposition above and the comments preceding it, is that $\T{G}(TS)$ is isomorphic to $\T{G}(ST)$ in the synchronous case. The question is still open if the same is true in the asynchronous case for those transducers one-way synchronizing transducers $T$ such that $\T{G}(ST)$ is infinite.

\section{Class of subgroups of \texorpdfstring{$V_n$}{Lg} containing conjugates of \texorpdfstring{$V_n$}{Lg}} \label{Sec: subgroups}
Let $\sim$ be an equivalence relation on $X_n^+$.  We call $\sim$ \emph{subgroup generating} if for all $\alpha, \beta\in X_n^+$, we have $\alpha\sim\beta$ if and only if for all $i\in X_n$, $\alpha i\sim \beta i$.  Using such relations, we produce subgroups of $V_n$ that, under some conditions, will be conjugate to $V_n$.  

Let $\sim$ be a subgroup generating relation on $X_n^*$ and let $v\in V_n$.  We say that the element $v$ \emph{preserves} $\sim$ if, in essence, $v$ maps prefixes to related prefixes. More specifically, $v$ preserves $\sim$ if whenever $(\alpha\chi)v=\beta\chi$ for some $\alpha,\beta\in X_n^+$ and all $\chi\in X_n^\omega$, we have $\alpha\sim\beta$. Note that if $v$ maps $\alpha\chi$ to $\beta\chi$ for some $\alpha,\beta\in X_n^+$ and all $\chi\in X_n^\omega$, then $v$ also maps $\alpha i\chi$ to $\beta i\chi$ $\chi\in X_n^\omega$ and $i\in X_n$.  This demonstrates the need for the forward implication in the definition of subgroup generating relations.

We now define the subgroup
\[V_{\sim}=\langle v\in V_n|\;v \mbox{ preserves }\sim\rangle.\]
Note that the identity trivially preserves $\sim$ and $V_{\sim}$ is closed under taking inverses. The reverse implication in the definition of subgroup generating guarantees $V_{\sim}$ is closed under composition.  

One method for producing subgroup generating relations is to use language accepting automata.  Let $A$ be an automaton $\langle X_n, Q, \pi, q_0\rangle$ with  transition function $\pi$, states $Q$, and start state $q_0$ such that $A$ has no inaccessible states. (This is for convenience.)  We can then define the relation $\sim_A$ as, for all $\alpha,\beta\in X_n^+$, $\alpha\sim_A\beta$ if and only if $\pi(\alpha,q_0)=\pi(\beta,q_0)$.  It is interesting to note that this implies equivalence classes under $\sim_A$ form regular languages. 

It is clear that for any automaton $A$, $\sim_A$ satisfies the forward implication for subgroup generating.  To ensure it satisfies the reverse implication, we make one additional constraint on $A$. We need that the ordered tuple $(\pi(0,q),\pi(1,q),\ldots,\pi(n-1,q))\neq(\pi(0,q'),\pi(1,q'),\ldots,\pi(n-1,q'))$ for each $q\neq q' \in Q$.  We call this \emph{Condition 1}.

Given an automaton $A = \gen{X_n, Q, \pi, q_0}$ which does not satisfies Condition 1, we can always construct one which does. We do this using the following procedure, which we call the \emph{collapsing procedure}.

\begin{enumerate}[label = (\arabic*)]
\item Define an equivalence relation on the states of $A$ by $q_1 \sim q_2$ if and only if for all $i \in X_n$ $\pi(i, q_1) = \pi(i, q_2)$. Let $[q_1], [q_2], \ldots [q_{r_1}]$ be the equivalence classes of this relation.
\item Form a new automaton $A_1= \gen{X_n, \{[q_0], [q_1], \ldots [q_{r_1}]\}, \pi_1, [q_0]}$ from $A$ as follows. The states of $A_1$ are the equivalence classes $[q_1], \ldots, [q_r]$. The transition function $\pi_1$ of $A_1$ is defined by $\pi_1(i, [q_1]) = [\pi(i, q_1)]$ for all $i \in  X_n$. Notice that the transition function is well defined since if $q \in [q_1]$ then $\pi(i, q) = \pi(i, q_1)$ for all $i \in X_n$. This represents one step of the collapsing procedure.
\item Check if $A_1$ satisfies condition 1. If yes we stop, otherwise we repeat the previous 2 steps with $A_1$ in place of $A$. 

\end{enumerate}

Since $A$ has finitely many states then there is a minimal $k \in \mathbb{N}$ such that $A_{k}$ satisfies condition 1. Therefore $\sim_{A_{k}}$ is a subgroup generating relation. Making use of $A_k$ we can now define a subgroup  generating relation $\sim_{A}$ which is equal to $\sim_{A_{k}}$. We do this as follows, for  all  $\alpha, \beta \in X_n^{+}$, $\alpha \sim_{A} \beta$ if and only if $\pi(\alpha, q_0)$ and $\pi(\beta, q_0)$ represent the same state of $A_{k}$, that is $\pi(\alpha, q_0)$ and $\pi(\beta, q_0)$ are identified at step $k$ of the collapsing procedure. Call $\sim_{A}$ the \emph{subgroup generating relation induced by $A$}. Notice that if $A_{k}$ is a single state transducer then $\sim_{A}$ has only one equivalence class, and so $V_{\sim_{A}} = V$.

The following lemma is proved in the appendix of \cite{BlkYMaisANav}.

\begin{lemma} \label{lemma:equivalent at k means transition alike on Xk}
Let $T= \gen{X_n, Q, \pi, \lambda }$ be a transducer. Let $A= \gen{X_n, Q, \pi, q_0}$ be the automaton obtained from $T$. Then $T$ is synchronizing at level $k$ if and only if after $k$ steps of the collapsing procedure the resulting transducer $A_{k}$ has one state. Moreover two states $q_1, q_2$ of $A$ are identified at step $i$ of the collapsing procedure if and only if for all $\Gamma \in X_n^{i}$ we have $\pi(\Gamma, q_1) = \pi(\Gamma, q_2)$.
\end{lemma}

Therefore given an initial, one-way synchronizing transducer $T_{q_0} = \gen{X_n, Q, \pi, \lambda}$ equal to its core  with inverse $S_{(\epsilon, q_0)} = \gen{X_n, Q, \pi', \lambda'}$, the subgroup generating relation induced by the automaton $A=\gen{X_n, Q, \pi',  q_0}$ obtained from $S_{(\epsilon,q_0)}$ has more than one equivalence class, in particular $A$ cannot be reduced to a single state transducer by the collapsing procedure.

We can now state the following theorem, which gives a description of the subgroup of $V_n$ formed by conjugating $V_n$ by a synchronous, one-way synchronizing transducer.

\begin{theorem}\label{thm:subgroup_synchronous_case}
Let $T_{q_0} = \gen{X_n,  Q, \pi, \lambda}$ be an initial, synchronous, one-way synchronizing transducer equal to core with inverse $S_{q_0^{-1}} = \gen{X_n, Q, \pi', \lambda'}$. Let $A= \gen{X_{n}, Q, \pi', \lambda'}$ be the automaton obtained from $S_{q_0^{-1}}$, then $V_n^{T_{q_0}}=V_{\sim_A}$. 
\end{theorem}
\begin{proof}
We first prove that $V_n^{T_{q_0}}\subseteq V_{\sim_A}$ by showing that conjugates of small swaps are contained in $V_{\sim_A}$. 

Let $\alpha\perp\beta\in X_n^+$ so that $(\alpha, \beta)\in V_n$ is a small swap and consider the action of $(\alpha, \beta)^{T_{q_0}}$ on $(\alpha)T_{q_0}\chi$.  We see
\begin{align*}
((\alpha)T_{q_0}\chi)(\alpha, \beta)^{T_{q_0}}&=(\alpha(\chi)S_{q^{-1}})(\alpha, \beta) T_{q_0}\\
&=(\beta(\chi)S_{q^{-1}})T_{q_0}\\
&=(\beta)T_{q_0}(\chi)S_{q^{-1}}T_{p}
\end{align*}
where $q=\pi(\alpha, q_0)$ and $p=\pi(\beta,q_0)$.  Since $S_{q^{-1}}T_{p}\in V_n$ by Proposition~\ref{pro:conjproducessubgroups}, $S_{q^{-1}}T_{p}$ will replace a prefix of $\chi$ before acting as the identity.  

Let the element $S_{q^{-1}}T_{p}\in V_n$ swap the prefixes $\gamma,\delta\in X_n^+$, i.e. $(\gamma\rho)S_{q^{-1}}T_{p}=\delta\rho$ for all $\rho\in X_n^\omega$.  This implies
\begin{align*}
((\alpha)T_{q_0}\gamma\rho)(\alpha, \beta)^{T_{q_0}}&=(\beta)T_{q_0}\delta\rho
\end{align*}
for all $\rho\in X_n^*$.

We now need to show $(\alpha)T_{q_0}\gamma\sim_A(\beta)T_{q_0}\delta$.  However, we know that $(\gamma\rho)S_{q^{-1}}T_{p}=\delta\rho$ and furthermore
$(\gamma\rho)S_{q^{-1}}=(\delta\rho)T_{p}^{-1}$.  Since this holds for all $\rho\in X_n^{\omega}$, this means $\pi'(\gamma, q)=\pi'(\delta, p)$ and therefore $\pi'((\alpha)T_{q_0}\gamma, q_0)=\pi'((\beta)T_{q_0}\delta, q_0)$.

With a similar proof for infinite words of the form $(\beta)T_{q_0}\chi$, this shows that $(\alpha, \beta)^{T_{q_0}}$ preserves $\sim_A$, as it trivially holds outside of its support.

To show $V_{\sim_A}\subseteq V_n^{T_{q_0}}$, we consider an element $v\in V_{\sim_A}$ such that for $\alpha,\beta\in X_n^+$ where $\alpha\sim_A \beta$ we have $(\alpha\chi)v=(\beta)\chi$ for all words $\chi\in X_n^\omega$.  

Let $\chi\in X_n^\omega$.  We then have
\begin{align*}
((\alpha)T_{q_0}^{-1}\chi)v^{T_{q_0}^{-1}}&=(\alpha(\chi)T_{q})v S_{q_0^{-1}}\\
&=(\beta(\chi)T_{q})S_{q_0^{-1}}\\
&=(\beta)T_{q_0}^{-1}(\chi)T_{q}S_{p^{-1}}
\end{align*}
where $q=\pi((\alpha)S_{q_0^{-1}},q_0)$ and $p=\pi'(\beta, q_0)$.  However, we know $\pi((\alpha)S_{q_0^{-1}},q_0)=\pi'(\alpha,q_0)$ and since $\alpha\sim_A\beta$, $q$ and $p$ represent equal states in $A_{k}$. 
This implies that for any word $\Gamma \in X_n^{k}$ we have $\pi'(\Gamma, q) = \pi(\Gamma, p)$. Therefore if $\Gamma$ is the length $k$ prefix of $\chi$ so that $\chi = \Gamma \Xi$, $\Xi \in X_n^{\omega}$, we have $\pi(\Gamma,q) = \pi'((\Gamma)T_{q},q ) = \pi'((\Gamma)T_{q},p)$. Let $r = \pi(\Gamma,q)$.  It now follows that 
\begin{align*}
((\alpha)T_{q_0}^{-1}\chi)v^{T_{q_0}^{-1}} &= (\beta)T_{q_0}^{-1}(\Gamma \Xi)T_{q}S_{p^{-1}}\\
&=(\beta)T_{q_0}^{-1}(\Gamma)T_{q}S_{p^{-1}}(\Xi)T_{r}S_{r^{-1}} \\
&= (\beta)T_{q_0}^{-1}(\Gamma)T_{q}S_{p^{-1}}\Xi.
\end{align*}
This demonstrates $v^{T_{q_0}^{-1}}\in V_n$ and therefore $V_n^{T_{q_0}}=V_{\sim_A}$
\end{proof}

\begin{example}\label{example:synchronizing}
Let $T=\langle X_2, \{a,b\}, \pi, \lambda \rangle$ be a two state transducer where $\lambda(i,a)=i$, $\lambda(i,b)=i+1 (\mod 2)$, $\pi(0,x)=a$, and $\pi(1,x)=b$. Note that $T_a$ is (one-way) synchronizing at level 1 and therefore $V^{T_a}\cong V$.  

Let $A=\langle X_2, \pi', \{a^{-1},b^{-1}, a^{-1}\}, a\rangle$ be the two state automaton with $\pi'(0,x)=x$ and $\pi'(1,x)\neq x$.  Note that $A$ is the inverse of $T$ without outputs. The relation $\sim_A$ has two equivalence classes, one with words containing an even number of ones, and other with an odd number of ones.  By the above arguments, $V^{T_a}=V_{\sim_A}$, which is the subgroup of $V$ that preserves the parity of ones in prefixes. 
\end{example}

As a remark, note that not every automaton $A$ satisfying Condition 1 produces a subgroup $V_{\sim_A}$ that is isomorphic to $V_n$.  We provide two interesting examples.

\begin{example}
Let $B=\langle X_2, \pi, \{a,b\}, a\rangle$ be the two state automaton with $\pi(0,a)=a$
and $\pi(i,x)=b$ in all other cases.  There are two equivalence classes of $\sim_B$, one for each state. One equivalence class consists of finite strings of 0s, and the other is its complement.  This implies $V_{\sim_B}$ is the stabilizer of $\overline{0}\in X_n^\omega$ in $V_2$ and is not isomorphic to $V_2$.  This can be generalised to find stablisers of any set of points.  This is related to \cite{GSstabilizers}, in which the authors examine stabilizers of points in Thompson group $F$.
\end{example}

\begin{example}
Let $C=\langle X_2, \pi, \{a,b\}, a\rangle$ be the two state automaton with $\pi(i,a)=b$
and $\pi(i,b)=a$ for all $i\in X_n$.  One equivalence class of $\sim_C$ consists of even length words, and the other of odd length words, which implies elements in $V_{\sim_C}$ preserve the parity of prefixes.  The group $V_{\sim_C}$ is equal to an embedded copy of $V_4$ in $V_2$ acting on the four-letter alphabet $\{00,01,10,11\}$.  Again, this can be generalised to embed $V_n$ in $V_2$.
\end{example}

Analogous subgroups of the generalised Thompson groups $F_n$ and $T_n$ can also be constructed using relations.  Several examples of which appeared in \cite{BowlinBrin_Fcolors}, although they arose in a different context.

We now generalise Theorem~\ref{thm:subgroup_synchronous_case} to include asynchronous transducers.  The proof is much longer and we divide it into several cases and subcases. 

\begin{theorem}\label{thm:subgroups_asynchronous_case}
Let $T_{q_0} = \gen{X_n, Q, \pi, \lambda}$ be an initial one-way synchronizing transducer equal to core with inverse $S_{(\epsilon, q_0)} = \gen{X_n, Q', \pi', \lambda'}$. Let $A= \gen{X_{n}, Q, \pi', \lambda'}$ be the automaton obtained from $S$, then $V_n^{T_{q_0}}=V_{\sim_A}$. 
\end{theorem}
\begin{proof}
We first prove that $V_n^{T_{q_0}}\subseteq V_{\sim_A}$ by showing that conjugates of small swaps are contained in $V_{\sim_A}$.

Let $\alpha \perp \beta \in X_n^{+}$ so that $(\alpha, \beta) \in V_n$ is a small swap. Let $u, v \in  X_n^{\ast}$ be such that $ \lambda(\alpha, q_0) = u$ and $\lambda(\beta, q_0) = v$. Also let $r = \pi(\alpha, q_0)$ and $t = \pi(\beta, q_0)$. The proof breaks up into two cases:\\
\emph{Case 1}: $(\alpha, \beta)^{T_{q_0}}$ 
 acts as a prefix replacement on some $u\delta$ or $v \delta$ for some $\delta \in X_n^{*}$\\
\emph{Case 2}: $(\alpha, \beta)^{T_{q_0}}$ acts as a prefix replacement on a proper prefix of $u$ or $v$.

We begin with Case 1.\\
\\
\emph{Proof of Case 1.} We first consider words $\rho \in \CCn$ such that  $(\rho)S_{(\epsilon, q_0)}$ has a prefix $\alpha$ and is therefore in the support of $(\alpha, \beta)^{T_{q_0}}$.  (The case when $(\rho)S_{(\epsilon, q_0)}$ has a prefix $\beta$ is analogous.)   Notice that all elements $\rho \in \CCn$ with a prefix $\alpha$  are such that $(\rho)T_{q_0}$ is in the set $u \im(T_{r}) = \{ u \chi : \chi \in \im(T_{r})\}$.
   
Let $\delta_1$ be a prefix of a word in $\im(T_r)$ such that there is a fixed $\mu \in X_n^{+}$ where $(\alpha, \beta)^{T_{q_0}}$ exchanges the prefix $u\delta_1$ with $\mu$, i.e. for any $\chi \in \CCn$, we have $(u\delta_1 \chi)(\alpha, \beta)^{T_{q_0}} = \mu \chi$. 
 
Let $\gamma \in X_n^{+}$ be a minimal word such that $\pi(\gamma, 
r) = \pi(\gamma, t)$ and also that $\lambda(\gamma, r) = \delta_1x $ 
for some $x \in X_n^{*}$. Such a word $\gamma$ exists because $T$ 
is synchronizing and $\delta_1$ is a prefix of a word in 
$im(T_r)$.  Let $q\in Q$ be equal to $\pi(\gamma, r)$ and $\phi = 
(\varphi)T_{q}$ for some $\varphi \in \CCn$. Observe that:
\begin{IEEEeqnarray}{rCl} \label{eqn: alpha beta}
 (u\delta_1 x \phi) S_{(\epsilon, q_0)}(\alpha, \beta)T_{q_0} &=& (\alpha \gamma \varphi) (\alpha, \beta) T_{q_0} \nonumber \\
 &=& (\beta \gamma \varphi)T_{q_0} \nonumber\\
 & =& v \lambda(\gamma, t) \phi\\
 & =& \mu x\phi. \nonumber
\end{IEEEeqnarray}
 
This demonstrates that $\lambda(\gamma,t) = \delta_2 x$ for some $\delta_2 \in X_n^{\ast}$ such that $v \delta_2 = \mu$. Therefore $(u\delta_1 x \phi) T_{q_0}^{-1}(\alpha, \beta)T_{q_0} =  v \delta_2 x \phi$.  

We need to argue that the states $\pi'(u\delta_1, (\epsilon, q_0))$ and $\pi'(v\delta_2, (\epsilon, q_0))$ are in the same equivalence class of $\sim_{A}$, since $(\alpha, \beta)^{T_{q_0}}$ swaps the prefixes $u\delta_1$ and $v\delta_2$.

We consider a trivial subcase.  If there exists a $\psi\in[u\delta_1]$ such that the word $(\psi)S_{(\epsilon,q_0)}$ does not have prefix $\alpha$ nor $\beta$, then $\psi$ is not in the support of $(\alpha, \beta)^{T_{q_0}}$.  This means $(\alpha, \beta)^{T_{q_0}}$ fixes $\psi$ and therefore forces $u \delta_1 = v \delta_2$.  Henceforth, we assume all elements of $[u\delta_1]$ have images under $S_{(\epsilon,q_0)}$ with prefix $\alpha$ or $\beta$.

The split this proof into the following subcases:\\
\emph{Subcase 1}: All elements of the set $([u\delta_1])S_{(\epsilon,q_0)}$ have prefix $\alpha$ and all elements of the set $([v\delta_2])S_{(\epsilon, q_0)}$ have  prefix $\beta$.\\
\emph{Subcase 2}: Some elements of the set $([u\delta_1])S_{(\epsilon,q_0)}$ have prefix $\alpha$ and others have prefix $\beta$.\\
Note that the subcase when some elements of the set $([v\delta_2])S_{(\epsilon,q_0)}$ have prefix $\alpha$ and others have prefix $\beta$ is similar to Subcase 2.\\
\\
\emph{Proof of Subcase 1}:
Suppose all elements of the set $([u\delta_1])S_{(\epsilon,q_0)}$ have prefix $\alpha$ and all elements of the set $([v\delta_2])S_{(\epsilon, q_0)}$ have  prefix $\beta$. Let $\pi'(u \delta_1, (\epsilon, q_0)) = (\T{S}(u\delta_1),r_1)$ and let $\pi(v\delta_2, (\epsilon, q_0)) = (\T{S}(v\delta_2), t_1)$ for some suffix $\T{S}(u\delta_1)$ and $\T{S}(v\delta_2)$ of $u \delta_1$ and $v \delta_2$ respectively. Let $\alpha \gamma_1 = \lambda'(u\delta_1, (\epsilon, q_0))$ and  $\beta \gamma_2 = \lambda'(v\delta_2, (\epsilon, q_0))$ for  prefixes $\gamma_1$ and $\gamma_2$ of $\gamma$.   Notice that if $\lambda(\alpha\tilde{\gamma}, q_0) = u \delta_1 y$ for some $y \in X_n^{\ast}$ and $\tilde{\gamma}$ such that $\pi(\tilde{\gamma}, r) = \pi(\tilde{\gamma}, t)$, then by equation $\eqref{eqn: alpha beta}$ we have that $\lambda(\beta \tilde{\gamma}, q_0) = v \delta_2 y$.

Now for all states $s$ of $T_{q_0}$, $\im(s)$ is clopen since 
$T_{q_0}$ represents a homeomorphism of  $\CCn$. Let  
$\tilde{\gamma}_1$ be such that $\lambda( \tilde{\gamma}_1, r_1) 
= \T{S}(u\delta_1) y$ for some $y \in X_n^{*}$, and such that 
$\pi(\gamma_1 \tilde{\gamma}_1, r) = \pi(\gamma_1 
\tilde{\gamma}_1, t)$. This means that 
$\lambda(\gamma_1\tilde{\gamma}_1, r) =\delta_1 y$, hence by the 
last sentence of the previous paragraph we have 
$\lambda(\gamma_1\tilde{\gamma}_1, t) = \delta_2 y$. This means that $\gamma_2$ is a prefix of $\gamma_1 \tilde{\gamma}_{1}$, however we shall not need this fact. Let $t_1' = 
\pi(\gamma_1, t)$. Let $q' = \pi(\tilde{\gamma}_1, r_1) = 
\pi(\tilde{\gamma}_1, t_1')$, $\im(q')$ is a union of finitely 
many basic open sets $\T{B}_{q'} $, we shall identify an element 
$B \in \T{B}_{q'} $ with the string $b \in X_n^{*}$ such that $B 
= [b]$. Let $w \in X_n^\ast$ be such that $w$ is a prefix of an element of $\im(q')$ and $w$ is  longer than  any $B \in 
\T{B}_{q'}$.  Then observe that $\pi'(u\delta_1 y w, (\epsilon, 
q_0)) = (w - \lambda(\T{L}_{q'}(w), q'), \pi(\T{L}_{q'}(w), 
q'))$. This is because $\T{L}_{q_0}(u\delta_1 y w)$ has $\alpha 
\gamma_1 \tilde{\gamma}_1$ as a prefix since $q_0$ is injective 
and $w \xi $ can be produced from $q'$ for any $\xi \in \CCn$. 
Now since by a previous observation we have that $\lambda(\beta 
\gamma_1 \tilde{\gamma}_1, q_0) =  v \delta_2 y$ and $\pi(\beta 
\gamma_1 \tilde{\gamma}_2, q_0) = q'$, we also have that $\pi'(v 
\delta_2 y w, (\epsilon, q_0)) = (\T{L}_{q'}(w), 
\pi(w-\T{L}_{q'}(w), q'))$ by the same argument as before. 
Therefore $\pi'(yw, (\T{S}(u\delta_1), r_1)) = \pi'(yw, 
\T{S}(v\delta_2), t_1))$.

Now observe that since $\T{L}_{q_0}(u \delta_1) = \alpha \gamma_1$, 
then if $G$ is the set of words $\tilde{\gamma}$ such that 
$\pi(\gamma_1 \tilde{\gamma}, r) = \pi(\gamma_1 \tilde{\gamma}, 
t)$ and $\lambda(\tilde{\gamma}, r_1) = \T{S}(u\delta_1) y$ for 
some $y \in X_n^{\ast}$ then every element of $[\delta_1]$ has a 
prefix in the set $\{ \lambda(\tilde{\gamma}, r_1) \}$. This 
follows since $q_0$ is surjective and any word $\Delta \in \CCn$ 
such that $(\Delta)T_{q_0}$ has a prefix $u \delta_1$, has a prefix 
$\alpha \gamma_1$. Set $Q(r_1) := \{ \pi(\tilde{\gamma}, r_1) : 
\gamma \in G \}$ and let $m'$  be the maximum length of a basic 
open set in a minimal finite covering by basic open sets of 
$\im(q')$ such that $q' \in Q(r_1)$. Let $m = m' + \max \{ 
|\lambda(\tilde{\gamma}, r_1)| - |\T{S}(\delta_1)| : \gamma \in G 
\}$ observe that for any $w \in X_n^{m}$ there is a word in 
$\im(r_1)$ which has $\T{S}(u\delta_1) w$ as a prefix. Therefore by 
arguments in the previous paragraph we see that  $\pi'(w, 
(\T{S}(u\delta_1), r_1)) = \pi'(w, (\T{S}(v\delta_2), t_1))$ for any 
$w \in  X_n^{m}$. Therefore $(\T{S}(v\delta_2), t_1)$ and 
$(\T{S}(u\delta_1), r_1)$ are related under the equivalence 
relation $\sim_{A}$.\hfill $\square$\\
\\
\emph{Proof of Subcase 2}:
Now we consider the case that all elements of the sets $([u\delta_1])S_{(\epsilon,q_0)}$, $([v\delta_2])S_{(\epsilon,q_0)}$ have  prefix $\beta$ or $\alpha$ and some element of the set $([u\delta_1])S_{(\epsilon,q_0)}$ has  prefix $\beta$ or some element of the set $([v\delta_2])S_{(\epsilon,q_0)}$ has  prefix $\alpha$ . Notice that by equation $\eqref{eqn: alpha beta}$, $(\alpha, \beta)^{T_{q_0}}$ acts on $[u \delta_1]$ substituting the prefix $u \delta_1$ with  $v\delta_2$ and vice-versa (since it has order 2).

Suppose some element of $([u\delta_1])S_{(\epsilon,q_0)}$ has  prefix $\beta$ (the other case is dealt with similarly). (Notice that by equation $(\eqref{eqn: alpha beta})$ there is an element of $([u\delta_1])S_{(\epsilon,q_0)}$ with prefix $\alpha$ and an element of $([u\delta_2] )S_{(\epsilon,q_0)}$ with  prefix $\beta$.)  This means that there is some $\zeta \in X_n^*$ such that $\lambda(\beta \zeta, q_0) = v v' \delta_1 y$ and $\pi(\beta \zeta, q_0) = \pi(\alpha \zeta, q_0) = p$ for a state $p$ of $T_{q_0}$ (since $T_{q_0}$ is synchronizing) for some $y \in X_n^{\ast}$ and some $v' \in X_n^{*}$ such that $vv' = u$. Hence for $\phi := \lambda(\varphi, p)$ $\varphi \in \CCn$ we have:

\begin{IEEEeqnarray}{rCl} \label{eqn: beta alpha}
 (u\delta_1 y \phi) T_{q_0}^{-1}(\alpha, \beta)T_{q_0} &=& (\beta \zeta \varphi) (\alpha, \beta) T_{q_0} \nonumber \\
 &=& (\alpha \zeta \varphi)T_{q_0} \nonumber\\
 & =& u \lambda(\zeta, r) \phi.
\end{IEEEeqnarray}

Since $(\alpha, \beta)^{T_{q_0}}$ acts as a prefix substitution 
on $[u \delta_1]$ replacing the prefix $u \delta_1$ with  
$v\delta_2$  and since $u = v v'$ and $p$ has no states of 
incomplete response we must have that $v'$ is a (possibly empty 
prefix of) $\delta_2$ and letting $\tilde{\delta}y := 
\lambda(\zeta, r)$ then $\delta_2 = v' \tilde{\delta}$ and $v 
\delta_2 = u\tilde{\delta}$. This now means that $\T{L}_{q_0}(u 
\delta_1) = \T{P}(\alpha, \beta)$ and $\T{L}_{q_0}(u 
\tilde{\delta}) =\T{P}(\alpha, \beta)$ where $\T{P}(\alpha, 
\beta)$ is the largest common prefix of $\alpha $ and $\beta$. 
Moreover from equations $\eqref{eqn: alpha beta}$ and 
$\eqref{eqn: beta alpha}$ we see that whenever there is some 
$\zeta ' \in X_n^{*}$ such that $\pi(\alpha \zeta, q_0) = \pi(\beta \zeta, q_0)$ and $\lambda(\alpha \zeta', q_0) = u 
\delta_1 y'$ (or $\lambda(\beta \zeta', q_0) = u \delta_1 y'$) for 
some $y' \in X_n^{\ast}$, then we have that $\lambda(\beta \zeta', q_0) =  u 
\tilde{\delta} y'$ (or $\lambda(\alpha \zeta', q_0) = u 
\tilde{\delta} y'$).

Let $\pi(\T{P}(\alpha, \beta), q_0)  = R$ then observe that since 
$\T{L}_{q_0}(u \delta_1) = \T{P}(\alpha, \beta)$ and 
$\T{L}_{q_0}(u \tilde{\delta})  =\T{P}(\alpha, \beta)$ then 
$\pi'(u \tilde{\delta}, (\epsilon,q_0)) = (\T{S}(u) 
\tilde{\delta}, R)$, $\pi'(u\delta_1, (\epsilon,q_0)) = 
(\T{S}(u)\delta_1, R)$ and $\lambda'(u \tilde{\delta}, 
(\epsilon,q_0)) = \lambda'(u\delta_1, (\epsilon,q_0)) = 
\T{P}(\alpha, \beta)$. Here $\T{S}(u) := u - 
\lambda(\T{P}(\alpha, \beta), q_0)$. Let $\T{S}(\alpha)$ and 
$\T{S}(\beta)$ be such that $\T{P}(\alpha ,\beta )\T{S}(\alpha) = 
\alpha$ and $\T{P}(\alpha, \beta)\T{S}(\beta) = \beta$.

Notice that since $\T{L}_{q_0}(u \delta_1) = \T{P}(\alpha, \beta)$ 
and $\T{L}_{q_0}(u \tilde{\delta})  =\T{P}(\alpha, \beta)$  and 
since $T_{q_0}$ is a homeomorphism, then from the state $R$ of 
$T_{q_0}$ it is possible to write any element of $[\T{S}(u)\delta_1]$ and $[\T{S}(u)\tilde{\delta}]$. Moreover  all inputs into 
$R$ which write  either $[\T{S}(u)\delta_1]$ or 
$[\T{S}(u)\tilde{\delta}]$ must have prefix $\T{S}(\alpha)$ or 
$\T{S}(\beta)$ by assumption.

Let $\zeta'$ be any element of $X_n^{*}$ such that $\pi(\zeta', 
r) = \pi(\zeta', t) = p$ for some state $p$ of $T_{q_0}$, and $\lambda(\alpha \zeta', q_0) = u 
\delta_1 y'$ for some $y' \in X_n^{\ast}$, by a previous 
observation, we also have $\lambda(\beta \zeta', q_0)  = u 
\tilde{ \delta} y'$. Let $\T{B}_{p}$ be a minimal finite cover by 
basic open sets for $\im(p)$ and let $w \in X_n^{*}$ be a prefix 
an element of $\im(p)$ which is longer than a basic open set in 
this cover (note again that we identify a basic open $B$ set with 
the string $b \in X_n^{\ast}$ such that $B = [b]$ ).

Consider $\pi'(u \delta_1 y' w, (\epsilon,q_0))$. This must be equal to $(w - 
\lambda'(\T{L}_{p}(w), p), \pi(\T{L}_{p}(w), p))$. This is 
because $\T{L}_{q_0}(u \delta_1 y' w)$ has $\alpha \zeta'$ as a 
prefix and since $T_{q_0}$ is injective and $[w]$ is a subset 
of $\im(p)$. Similarly $\pi'(u \tilde{\delta} y' w, q_0) = (w - 
\lambda'(\T{L}_{p}(w), p), \pi(\T{L}_{p}(w), p))$, since 
$\T{L}_{q_0}(u\tilde{\delta} y' w)$ has $\beta \zeta'$ as a 
prefix. Therefore $\pi'(u \delta_1 y' w, q_0) = \pi'(u 
\tilde{\delta} y' w, q_0)$.

From $R$ it is possible to write any element of $[\T{S}(u)\delta_1 
]$ and $[\T{S}(u)\tilde{\delta}]$. Moreover  all inputs into 
$R$ which write  either an element of $[\T{S}(u)\delta_1 
]$ or an element of 
$[\T{S}(u)\tilde{\delta}]$ must have prefix $\T{S}(\alpha)$ or 
$\T{S}(\beta)$. It was also observed above that whenever there is 
some $\zeta  \in X_n^{*}$ such that $\lambda(\alpha \zeta, q_0) 
= u \delta_1 y'$ (or $\lambda(\beta \zeta, q_0) = u \delta_1 y'$) 
for some $y' \in X_n^{\ast}$  and $\pi(\alpha \zeta, q_0) = 
\pi(\beta \zeta, q_0)$ then we have that $\lambda(\beta \zeta, 
q_0) =  u \tilde{\delta} y'$ (or $\lambda(\alpha \zeta, q_0) = u 
\tilde{\delta} y'$. Furthermore since $T_{q_0}$ has finitely many 
states there is a maximum size for the length of a basic open set 
in a minimal cover by basic open sets of the image of any state 
of $T_{q_0}$. These things together with the computation above 
now means that there is an $M \in \mathbb{N}$ such that for all 
words $w \in X_n^{M}$ we have $\pi(u\delta_1 w, q_0) = \pi(u 
\tilde{\delta} w, q_0)$. Therefore we have that 
$(\T{S}(u)\tilde{\delta}, R) = \pi'(u\tilde{\delta}, (\epsilon, 
q_0))$ and $(\T{S}(u)\delta_1, R) = \pi'(u \delta_1, 
(\epsilon, q_0))$ are related under the equivalence relation 
$\sim_A$. \hfill $\square$\\

The above arguments now deal with the case where $(\alpha, \beta)^{T_{q_0}}$ acts as a prefix replacement on $[u \delta] $ for some (possibly empty) prefix $\delta$ of $\im(r)$. Swapping the roles of $\alpha$ and $\beta$ above, this also covers the case where $(\alpha, \beta)^{T_{q_0}}$ acts as a prefix replacement on $[v \delta]$ for some (possibly empty) $\delta$ prefix of $\im(t)$. \hfill $\square$\\
\\
\emph{Proof of Case 2}:
We now deal with the case where $(\alpha, \beta)^{T_{q_0}}$ acts as a prefix replacement on  $[\overline{u}]$ for some prefix $\overline{u}$ of $u$ such that $u = \overline{u}\underline{u}$ (the case of a prefix of $v$ is dealt with analogously). If $\bar{u} = u$ then we are in the previous case, so we may assume that $\bar{u}$ is a proper prefix of $u$ and so $\underline{u}$ is not the empty word.

Observe that all elements of the set $([\overline{u}])T_{q_0}^{-1}$ must have a prefix equal to $\alpha$ or $\beta$ otherwise  $(\alpha, \beta)^{T_{q_0}}$ acts as the identity on $[\overline{u}]$. This is a consequence of the dynamics of finite order elements of the groups $V_n$.

If all elements of $([\overline{u}])S_(\epsilon,q_0)$ have prefix equal $\alpha$ then $\overline{u} = u$ (otherwise $T_{q_0}$ is not injective since $\lambda(\alpha, q_0) = u$). Hence we have a contradiction.

Therefore we may assume that elements of the set 
$([\overline{u}])S_(\epsilon,q_0)$ all have prefixes equal to 
$\alpha$ or $\beta$ and there exists an element with a prefix 
$\alpha$ (since $\lambda(\alpha, q_0) = u$) and one with a 
prefix  $\beta$. Now since there is an element of 
$([\overline{u}])S_{(\epsilon,q_0)}$ with prefix, $\beta$ we must 
have that $v = \lambda(\beta, q_0)$ contains $\overline{u}$ as a 
prefix, since otherwise  $\overline{u} = v \delta'$ for some 
$\delta' \in X_n^{*}$ and $(\alpha, \beta)^{T_{q_0}}$ acts as a 
prefix substitution on $[v \delta' ]$ and we have reduced to the 
previous case.  Therefore $\T{L}_{q_0}(\overline{u}) = 
\T{P}(\alpha, \beta)$ the greatest common prefix of $\alpha$ and 
$\beta$. Let $R := \pi(\T{P}(\alpha, \beta), q_0)$, and let 
$\overline{u}_1 = \lambda(\T{P}(\alpha, \beta), q_0)$. Let 
$\T{S}(\alpha)$ and $\T{S}(\beta)$ be such that $\T{P}(\alpha, 
\beta) \T{S}(\alpha) = \alpha$ and $\T{P}(\alpha, \beta) 
\T{S}(\beta) = \beta$. Since $\T{L}_{q_0}(\overline{u}) = 
\T{P}(\alpha, \beta)$ and $R := \pi(\T{P}(\alpha, \beta), q_0)$, 
then we must have that $([\T{S}(\alpha) ])T_{R} \cup 
([\T{S}(\beta) ])T_{R} = [\overline{u}_2]$, where $\overline{u} = 
\overline{u}_{1}\overline{u}_{2}$. Therefore  
$\lambda(\T{S}(\alpha), R) = \overline{u}_{2}\underline{u}$ and 
$\lambda(\T{S}(\beta), R) = \overline{u}_{2} v'$ for some $v' \in 
X_n^{*}$ such that $v= \bar{u}v'$. 

Now suppose  $v' = \epsilon$. Since $(\alpha, 
\beta)^{T_{q_0}}$ acts as a prefix substitution on 
$[\overline{u}]$, and since $\lambda(\beta, q_0) = \overline{u}$, 
then for any $\psi \in \CCn$ we must have $\lambda(\psi, t) = 
\lambda (\psi,r)$ or the set $\{\lambda(\psi, t)\}$ has a 
non-empty prefix. This is because for any $\psi \in \CCn$ we have:

\begin{IEEEeqnarray}{rCl} 
 (u (\psi)T_{r}) S_{(\epsilon, q_0)}(\alpha, \beta)T_{q_0} &=& (\alpha \psi) (\alpha, \beta) T_{q_0} \nonumber \\
 &=& (\beta \psi)T_{q_0} \nonumber\\
 & =& \overline{u}(\psi)T_{t} 
\end{IEEEeqnarray}

and  $(\alpha, \beta)^{T_{q_0}}$ acts as a prefix substitution on $[\overline{u}]$. In the case that the set $\{\lambda(\psi, t): \psi \in \CCn\}$ has a non-empty prefix, then $t$ is a state of incomplete response and we have a contradiction. If for all $\psi \in \CCn$ we have $\lambda(\psi, t) = \lambda (\psi,r)$ then $t = r$, which forces $\underline{u} = \epsilon$ and so $u = \overline{u}$ which is a contradiction.

Therefore we may assume that $v' \ne \epsilon$. Now since 
 $([\T{S}(\alpha) ])T_{R} \cup ([\T{S}(\beta) ])T_{R} = [\overline{u}_2]$, and $\lambda(\T{S}(\beta), R) = \overline{u}_{2} v'$ and $\lambda(\T{S}(\alpha), R) = \overline{u}_{2} \underline{u}$ where $v' \ne \epsilon$ and $\underline{u} \ne \epsilon$ . Then we must have that $X_n = \{0,1\}$ and $v' =i$ and $\underline{u} = \bar{i}$ where $i \in \{0,1\}$ and $\bar{i} = i+1 \mod 2$. Moreover in this case $r$ and $t$ must be homeomorphism states. However this now means that $(\alpha, \beta)^{T_{q_0}}$ acts on $\overline{u}\bar{i} (\psi)T_{r}$ for any $\psi \in \CCn$ by sending it to $\overline{u}i (\psi)T_{t}$. Therefore since $(\alpha, \beta)^{T_{q_0}}$  acts as a prefix substitution on $[\overline{u}]$ we have that $t$ either has incomplete response or is equal to $r$, and so we conclude that $t= r$. It now follows that $\overline{u}\bar{i} (\psi)T_{r} \mapsto \overline{u}i (\psi)T_{r}$ under $(\alpha, \beta)^{T_{q_0}}$. Therefore $(\alpha, \beta)^{T_{q_0}}$ does not act on $[\overline{u}]$ as a prefix substitution, which is a contradiction.
 
 This covers the case where $(\alpha, \beta)^{T_{q_0}}$ acts  as a prefix replacement on  $[\overline{u}]$ for some prefix  $\overline{u}$ of $u$ such that $u = \overline{u}\underline{u}$.  By swapping the roles of $\alpha$ and $\beta$ we can repeat the  arguments for the case where $(\alpha, \beta)^{T_{q_0}}$ acts  as a prefix replacement on  $[\overline{v}]$ for some prefix   $\overline{v}$ of  $v$. \hfill $\square$\\

We now free all symbols used above.

To show $V_{\sim_A}\subseteq V_n^{T_{q_0}}$, we consider an element $v\in V_{\sim_A}$ such that for $\alpha,\beta\in X_n^+$ where $\alpha\sim_A \beta$ we have $(\alpha\chi)v=\beta\chi$ for all words $\chi\in X_n^\omega$. We shall argue that $v^{T_{q_0}^{-1}}$ is an element of $V_{n}$.

Since $ \alpha \sim_{A} \beta$, let $k$ be minimal so that for 
all words $\Delta \in X_n^{k}$ we have $\pi' (\Delta, 
\pi'(\alpha, (\epsilon, q_0))) = \pi' (\Delta, \pi'(\beta, 
(\epsilon, q_0)))$, such a $k$ exists by Lemma 
\ref{lemma:equivalent at k means transition alike on Xk}. Let $j 
\in \mathbb{N}$ be such that for all states $(w, p')$ of 
$S_{(\epsilon, q_0)}$ and for all words $\Gamma' \in X_n^{j}$ we 
have $|(\Gamma')S_{(w,p')}| \ge 1$. Let $i \in \mathbb{N}$ be 
such that for all states $p$ of $T_{q_0}$, and for all words 
$\Gamma \in X_n^{i}$ we have $|(\Gamma)T_{p}| \ge  \max\{|\alpha| 
+ k, |\beta| + k, j \}$ . 

Let $u \in X_n^{i}$ be an arbitrary 
word of length $i$ such that $(u)T_{q_0} = \alpha x$ for some $x 
\in X_n^{*}$, $|x| \ge k$. Observe that by choice of  $i$ we have 
$(\alpha x)S_{(\epsilon, q_0)} = u_1$ for some non-empty prefix 
$u_1$ of $u$. Since $u_1$ is a prefix of $u$ we may write $u = 
u_1 u_2$ for some $u_2 \in X_n^{*}$. Let $r := \pi(u, q_0)$, and 
let  $((\alpha x)_2, s) := \pi'(\alpha x, (\epsilon, q_0))$. By 
the inversion algorithm $(\alpha x)_2$ is a suffix of $\alpha 
x$.

Now since $T_{q_0}S_{(\epsilon, q_0)} = \id$ it must be the case that for any $\delta$ in $\im(r)$ we have $\lambda'( \delta, ((\alpha x)_2,s)) = u_2 \delta$. This is because $\lambda( u, q_0) = \alpha x$, $r := \pi(u, q_0)$, $((\alpha x)_2, s) := \pi'(\alpha x, (\epsilon, q_0))$ and $\lambda'(\alpha x, (\epsilon, q_0)) = u_1$. Consider the following for arbitrary $\gamma \in \C{C}{n}{}$

\begin{IEEEeqnarray}{rCl}
(u \gamma) T_{q_0} v S_{(\epsilon,q_0)} &=& (\alpha x (\gamma)T_{r}) v S_{(\epsilon,q_0)} \\
&=& (\beta x (\gamma)T_{r})S_{(\epsilon,q_0)} \\ 
&=& (\beta x) S_{(\epsilon, q_0)} (\gamma)T_{r}S_{((\alpha x)_2,s)} \label{eqn:line3} \\
&=& (\beta x) S_{(\epsilon, q_0)} u_2 \gamma.
\end{IEEEeqnarray}

Line $\eqref{eqn:line3}$ in the equation above follows since $\pi'(\alpha \Delta, (\epsilon, q_0)) = \pi'(\beta \Delta, (\epsilon, q_0))$ for any word $\Delta \in X_n^{k}$, in particular, $\pi'(\alpha x, (\epsilon, q_0)) = \pi'(\beta x, (\epsilon, q_0))$. Therefore we have:

\begin{equation*}
(u \gamma) T_{q_0} v S_{(\epsilon,q_0)} = (\beta x) S_{(\epsilon, q_0)} u_2 \gamma
\end{equation*}

Now a similar argument shows that for an $u \in X_n^{i}$ such that $(u)T_{q_0} = \beta x $ for some $x \in X_n^{*}$ of length at least $k$ (by choice of $i$), we also have that, for any $\gamma \in X_n^{k}$, $(u\gamma) T_{q_0} v S_{(\epsilon,q_0)} = \chi \gamma$ for some $\chi \in X_n^{+}$.

Now since $T_{q_0} v S_{(\epsilon,q_0)}$ acts trivially on all cones $[u]$ where $u \in X_n^{i}$,  the above demonstrates $v^{T_{q_0}^{-1}}\in V_n$ and therefore $V_n^{T_{q_0}}=V_{\sim_A}$.

\end{proof}

\section{Class of overgroups of \texorpdfstring{$V_n$}{Lg} containing conjugates of \texorpdfstring{$V_n$}{Lg}}\label{Section: overgroups}

\subsection{Synchronous case}
Let $T$ be a synchronous transducer and let $\T{G}(T)$ be the automaton group generated by its states.  For an element $g\in \T{G}(T)$ and finite word $\alpha\in X_n^*$, we define the homeomorphism $[\alpha]_g$ to act as $g$ on the cone $[\alpha]$ and trivially otherwise.  More specifically, for all $\chi\in X_n^\omega$, $(\alpha\chi)[\alpha]_g=\alpha(\chi)g$ and $[\alpha]_g$ fixes points without the prefix $\alpha$.

We can now define the group of homeomorphisms \[V_n({\T{G}(T)})=\langle V_n,\{[\alpha]_g|\alpha\in X_n^*,g\in \T{G}(T)\}\rangle.\]
See \cite{Nek04} for many more details.  To simplify notation, we will use $V_n(T)$ to mean $V_n({\T{G}(T)})$ as the self-similar group generated by the states of $T$ is understood. 

\begin{theorem} \label{Theorem: ovegroups synchronous}
Let $T=\langle X_n, Q, \pi, \lambda\rangle$ be a minimal, synchronous, synchronizing transducer  with inverse $S=\langle X_n, Q^{-1}, \pi', \lambda'\rangle$.  Then $V_n^{T_{q_0}^{-1}}=V_n(TS)$.
\end{theorem}
\begin{proof}
We first prove $V_n^{T_{q_0}^{-1}}\subseteq V_n(TS)$ by showing that conjugates of small swap (which generate $V_n$) are contained in $V_n(TS)$.

Let $\alpha\perp\beta\in X_n^+$ so that $(\alpha, \beta)\in V_n$ is a small swap.  Let $\chi\in X_n^{\omega}$ and consider the action of  $(\alpha, \beta)^{T_{q_0}^{-1}}$ on the word $(\alpha)S_{q_0^{-1}}\chi$.
\begin{align*}
((\alpha)S_{q_0^{-1}}\chi)(\alpha, \beta)^{T_{q_0}^{-1}}&=(\alpha(\chi)T_{q})(\alpha, \beta) S_{q_0^{-1}}\\
&=(\beta(\chi)T_{q})S_{q_0^{-1}}\\
&=(\beta)S_{q_0^{-1}}(\chi)T_{q}S_{p^{-1}}
\end{align*}
where $q=\pi((\alpha)S_{q_0^{-1}}, q_0)=\pi'(\alpha,q_0)$ and $p=\pi'(\beta, q_0)$. For brevity, let $g=T_{q}S_{p^{-1}}$ and $h=T_{p}S_{q^{-1}}$.  With a similar result for words of the form $(\beta) S_{q_0^{-1}}\chi$, we see that $(\alpha, \beta)^{T_{q_0}^{-1}}$ acts as a small swap before acting as the transducers $g$ and $h$ and therefore
\[(\alpha, \beta)^{T_{q_0}^{-1}}=((\alpha)S_{q_0^{-1}}, (\beta)S_{q_0^{-1}})[(\beta)S_{q_0^{-1}}]_{g}[(\alpha)S_{q_0^{-1}}]_{h}\in V_n(TS)\]
This shows $V_n^{T_{q_0}^{-1}}\subseteq V_n(TS)$.

For the reverse containment, we show generators of $V_n(TS)$ are taken into $V_n$ by conjugation by $T_{q_0}$. From Proposition~\ref{pro:conjproducessubgroups}, we have that $(\alpha, \beta)^{T_{q_0}}\in V_n$.  We must now show that $[\alpha]_g^{T_{q_0}}\in V_n$ for all $\alpha\in X_n^*$ and $g\in \T{G}(TS)$.

Let $\alpha\in X_n^*$, $\chi\in \CCn$, and $g=T_q S_{p^{-1}}\in \T{G}(TS)$ for some states $q,p\in Q$.  Consider the action of $[\alpha]_g^{T_{q_0}}$ on the word $(\alpha)T_{q_0}\chi$. We see
\begin{align*}
((\alpha)T_{q_0}\chi)[\alpha]_g^{T_{q_0}}&=(\alpha(\chi)S_{s^{-1}})[\alpha]_g T_{q_0}\\
&=(\alpha(\chi)S_{s^{-1}}g)T_{q_0}\\
&=(\alpha)T_{q_0}(\chi)S_{s^{-1}}gT_{t}
\end{align*}
where $s=pi'((\alpha)T_{q_0}, q_0)=\pi(\alpha, q_0)=t$.  Recalling the definition of $g$, we have $S_{s^{-1}}gT_{t}=S_{s^{-1}}T_q S_{p^{-1}}T_{t}\in V_n$ by Proposition~\ref{pro:conjproducessubgroups}.
Therefore $V_n^{T_{q_0}^{-1}}=V_n(TS)$.
\end{proof}

\begin{example}\label{example:synchronizingovergroup}
We revisit the synchronous one-way synchronizing transducer in Example~\ref{example:synchronizing}.  Recall that $T$ is a two state transducer with states $a$ and $b$ and thus has inverse $S$ with states $a^{-1}$ and $b^{-1}$. The product transducer $TS$ has four states $aa^{-1}$, $bb^{-1}$, $ab^{-1}$, and $ba^{-1}$.  It is clear that states $aa^{-1}$ and $bb^{-1}$ produce the identity homeomorphism. With some calculation, the states $ab^{-1}$ and $ba^{-1}$ produce identical homeomorphisms of $\CCn$, which simply change each 0 to a 1 and vice versa.  Therefore $\T{G}(TS)$ is a group of size two.

The Nekrashevych group $V_2(TS)$ is the overgroup of $V_2$ in which, after acting as a finite prefix exchange, elements will either act trivially on infinite suffixes or switch all 0s and 1s.  
\end{example}

Theorem~\ref{Theorem: ovegroups synchronous} is a generalisation of the results of \cite{BDJVnF}, in which the authors describe a case of Theorem~\ref{Theorem: ovegroups synchronous}.  Specifically, they consider the case when the self-similar group $\T{G}(TS)$ acts as a semiregular permutation group on each letter of words in $\CCn$.  The overgroup in Example~\ref{example:synchronizingovergroup} is an example of such groups.

\subsection{Asynchronous Case}

As mentioned in the introduction, if a transducer $T$ is asynchronous, it must have non-homeomorphism states.  This provides difficulty when approaching the `automaton group' generated by the states of $T$'.  However, if $T$ does have a homeomorphism state, i.e. there exists a state $q_0$ of $T$ such that $T_{q_0}$ is a homeomorphism, then for any state $q$ of $T$ accessible from  $q_0$, $T_q$ is injective and has clopen image. We use this fact to construct homeomorphisms from the non-surjective states by `filling in the holes' left by the non-surjectivity. This process  was introduced in \cite{BlkYMaisANav} for synchronizing transducers. The groups of homeomorphisms produced in this way help us generalise Theorem~\ref{Theorem: ovegroups synchronous} for asynchronous transducers. 

The following definition is natural to make in light of the results of this section:

\begin{definition}\label{Def: partially invertible}
Let $T$ be a finite transducer all of whose states are injective and have clopen image, then $T$ is called \emph{partially invertible}. 
\end{definition}
Let $\mathfrak{P}_n$ denote the set of all finite partially invertible transducers over the alphabet $X_n$. Note that every invertible transducer is partially invertible. The following proposition shows that $\mathfrak{P}_n$ is closed under taking products of transducers. 

\begin{proposition}\label{Prop: set of partially invertible transducers is closed}
The set $\mathfrak{P}_{n}$ is closed under the usual transducer product defined in Section~\ref{Section: preliminaries}.
\end{proposition}
\begin{proof}
Let $T, S \in \mathfrak{P}_{n}$. Let $t$ be an arbitrary state of $T$ and let $s$ be an arbitrary state of $S$. Consider the state $(t,s)$ of the product $TS$. The map $TS_{(t,s)}: \CCn \to \CCn$ corresponds to the composition of the functions $T_{t} \circ T_{s}$. 

Now as $\im(t)$ is clopen subset of $\CCn$ it may be written as a union of finitely many basic open. Let $\im(t) = \sqcup_{1 \le i \le k} [\alpha_i]$ for some $k \in \N$ and $\alpha_i \in X_n^{*}$.
Fix $1 \le i \le k$ and consider the cone $[\alpha_i]$. Let $s' := \pi_{S}(\alpha_i, s)$. Then $([\alpha_i])T_{s} = \lambda_{S}(\alpha_i, s) \im(s')$ Now since $s'$ has clopen image it follows that $\lambda_{S}(\alpha_i, s) \im(s')$ is also clopen.

Since $i$ was arbitrary, then for any $1\le j \le k$ we have that $([\alpha_j])T_{s}$ is clopen. Since the union of finitely many clopen sets is also clopen it follows that $(\im(t))T_{s}$ is clopen. Therefore $(\CCn)T_{t}T_{s}$ has clopen image and so $TS_{(t,s)}$ also has clopen image. Furthermore since $T_{t}$ and $T_{S}$ are injective $T_{t}\circ T_S$ is also injective, therefore $(t,s)$ is an injective state of $TS$.

Since $(t,s)$ where arbitrary states of $T$ and $S$ respectively, we therefore have that $TS \in \mathfrak{P}_{n}$. This concludes the proof.
\end{proof}

We fix from henceforth $T \in \mathfrak{P}_{n}$. 

The following definition provides a framework from which we build homeomorphisms.  Note that for every non-homeomorphism state $p$ of $T$, $T_p$ is injective and has a clopen image by assumption.

\begin{definition}
Let $T$ be a finite partially invertible transducer . A 
pair of tuples \[((\rho_1, \ldots ,\rho_l),(p_1, \ldots, p_l))\] 
for not necessarily distinct states 
$p_1, \ldots, p_l\in Q$ for $l \in \N$ and not necessarily distinct or 
incomparable words $\rho_1, \ldots ,\rho_l \in X_n^{\ast}$, is 
called a \emph{ viable combination} if $\bigcup_{1 \le i \le l} 
\rho_i \im(p_i) = \CCn$, and for $i \ne j$, $\rho_i \im(p_i) \cap 
\rho_j \im(p_j) = \emptyset$. We call $l$ the \emph{length of the viable combination}. The viable combination $((\rho_1, \ldots ,\rho_l),(p_1, \ldots, p_l))$ is called \emph{effective for $\CCn$} if $l \equiv 1 \mod{n-1}$.
\end{definition}

Let $((\rho_1, \ldots ,\rho_l),(p_1, \ldots, p_l))$ be a viable combination for $T$. If the viable combination is effective for $\CCn$, we may find a a maximal anti-chain of $\{\eta_1, \ldots, \eta_{l}\}$  $\CCn$. We can now define a self-homeomorphism $h$ of $\CCn$ so that $h$ maps $\eta_i\chi$ to $\rho_i (\chi)T_{p_i}$.  In essence, $h$ acts by replacing the prefix $\eta_i$ with $\rho_i$ and then as the transducer $T$ in state $p_i$.  It follows from the definition of a viable combination and the fact that $\{\eta_1, \ldots, 
\eta_{l}\}$ is a maximal anti-chain that $h$ is surjective and injective.  We call any homeomorphism $h$ constructed in this way a \emph{completion of $T$} as used in \cite{BlkYMaisANav}. Notice that if $l \cong r \mod{n-1}$ for some $1 \le r \le n-1$, then there is a maximal antichain of length $l$ in $\CCnr$, thus we may construct a self-homeomorphism $h$ of $\CCnr$ as before. In this case we also call $h$ completion of $T$. Such completions will not play a role in this article.

\begin{remark}
If $T$ is a partially invertible transducer then $T$ has at least one viable combination. For suppose  $p$ is a state of $T$ and $\eta \in X_n^{*}$ is such that $[\eta] \subseteq \im(p)$. Let $j \in \N$ be such that for all $\gamma \in X_n^{j}$ $|\lambda_{\gamma, p}|\ge |\eta|$ and let $\{\varphi_{i} \mid 1 \le i \le l\}$ be the maximal subset of $X_n^{j}$ for which $([\varphi_{i}])T_{p} \subseteq [\eta]$ for all $1 \le i \le j$. For $1 \le i \le j$, let $p_i = \pi_{T}(\varphi_i, p)$, and let $\rho_i \in X_n^{\ast}$ be such that $\lambda_{T}(\varphi, p) = \eta \rho_i$. Then the pair $((\rho_1, \ldots, \rho_l), (p_1, \ldots, p_l))$ is a viable combination of $T$.
\end{remark}

For a transducer $T \in \mathfrak{P}_{n}$ let \[\mathcal{H}(T):= \gen{ h \in H(\CCn) \mid h \mbox{ is a completion of } T } \] be the subgroup of $H(\CCn)$ generated by all completions of $T$, then observe that by the remark above $\T{H}(T)$ is non-empty precisely when there is a viable combination of $\T{H}(T)$ of length congruent to $1$ modulo $n-1$.

Observe that if $R$ is the $\omega$-minimal transducer representing $T$ then $\mathcal{H}(T) = \mathcal{H}(R)$. This is because as states of $R$ correspond to $\omega$-equivalence classes of states of $T$ then any viable combination of $T$ may be realised as a viable combination of $R$ and vice versa. Therefore if $T_1$ and $T_2$ are equivalent transducers, that is the $\omega$-minimal transducer representing $T_1$ is isomorphic to the $\omega$-minimal transducer representing $T_2$ then $\mathcal{H}(T_1) = \mathcal{H}(T_2)$. Let \[\widetilde{\mathfrak{P}}_{n} = \{ \{T\}_{\omega} \mid T \in \mathfrak{P}_{n}\}\] then by Remark~\ref{Remark: some comments about product} part \ref{Remark: transducer product is associative on equivalent transducers and respects equivalence classes} and Proposition~\ref{Prop: set of partially invertible transducers is closed} we have that $\widetilde{\mathfrak{P}}_{n}$  is a monoid.

If the transducer $T$ (as above) is synchronous, then a completion $h$ of $T$  looks like an element of $V_n(T)$.  Specifically, it is the element of $V_n$ that maps the prefix $\eta_i$ to $\rho_i$ for each $i$ multiplied by $\Pi[\rho_i]_{T_{p_i}}$.  This implies that for synchronous $T$, the group generated by completions of $T$ is a subgroup of $V_n(T)$.   
 
If $T$ is asynchronous, partially invertible and at least one of the states $\{p_1,\ldots,p_l\}$ is a non-homeomorphism, then $\{\rho_1,\ldots,\rho_l\}$ cannot be a maximal anti-chain and the previous decomposition of $h$ in a element of $V_n$ and $\T{G}(T)$ is not well-defined.  However, the following proposition shows that $\mathcal{H}(T)$ is still an overgroup of $V_n$, when $T$ satisfies a straightforward condition.

\begin{proposition}\label{proposition: subgroup of completions}
Let $T\in \mathfrak{P}_n$ and let $q$ be a state of $T$ such that $(\chi)T_q=\chi$ for all $\chi\in\CCn$.  Then $V_n\leq \T{H}(T)$.
\end{proposition}
\begin{proof}
Let $T$ and $q$ be defined as in the statement of the proposition.  Also let $v\in V_n$.  This implies there are maximal antichains $\{\alpha_1,\ldots,\alpha_l\}$ and $\{\beta_1,\ldots,\beta_l\}$ of $\CCn$ such that $(\alpha_i\chi)v=\beta_i\chi$ for all $\chi\in\CCn$.  However, $((\alpha_1,\ldots,\alpha_l)(q,\ldots,q))$ is a viable combination since $\text{im}(q)=\CCn$.  This shows that $v$ is a completion of $T$.
\end{proof}

This demonstrates how $\T{H}(T)$ can be seen as a generalisation of $V_n(T)$, combining the action of states of $T$ with finite prefix exchanges. Therefore, when $T$ has a state that induces the identity homeomorphism, we extend the notation $V_n(T):= \T{H}(T)$.  Note that by the comments in the paragraph preceding Proposition~\ref{proposition: subgroup of completions}, this agrees with the synchronous case.  

As an immediate corollary to Proposition~\ref{proposition: subgroup of completions}, we have the following.

\begin{corollary}
Let $T_{q_0}$ be an initial invertible transducer with inverse $S_{(\epsilon, Q_0)}$.  Then $V_n\leq \T{H}(TS)$.
\end{corollary}
\begin{proof}
Firstly, note that $T,S\in\mathfrak{P}_n$ and therefore $TS\in\mathfrak{P}_n$.  Also, since $S_{(\epsilon, Q_0)}$ is the inverse of $T_{q_0}$, the state $(q_0,(\epsilon, q_0))$ induces the identity homeomorphism in $TS$. By Proposition~\ref{proposition: subgroup of completions}, $V_n\leq \T{H}(TS)$.
\end{proof}

It is interesting to note that in \cite{LodhaMooreGroup}, Lodha and Moore also construct a group by combining Thompson group $F$ with a rational homeomorphism of $\mathfrak{C}_{2}$ that is asynchronous.  This is related to but not the same construction we present here.  

Before we state Theorem \ref{theorem: overgroups asynchronous}, which is the generalization of Theorem \ref{Theorem: ovegroups synchronous} to asynchronous transducers, we require the following corollary of Proposition~\ref{Prop: set of partially invertible transducers is closed}.

\begin{corollary}
Let $T$ be an invertible, one-way synchronizing transducer with inverse $S$. Then $TS \in \mathfrak{P}_{n}$
\end{corollary} 
\begin{proof}
Observe that by virtue of having a homeomorphism state $q_0$ such that all states of $T$ are accessible from $q_0$  it follows that $T \in \mathfrak{P}_{n}$. In a similar way $S$ has a state $(\epsilon, q_0)$ so that $S_{(\epsilon, q_0)}$ represents the inverse homeomorphism of $T_{q_0}$, and all states of $S$ are accessible from $(\epsilon, q_0)$, hence $S \in \mathfrak{P}_{n}$. It therefore follows that $TS \in \mathfrak{P}_{n}$ by Proposition~\ref{Prop: set of partially invertible transducers is closed}.
\end{proof}

The corollary above implies that for an invertible, one way synchronizing transducer $T$ with inverse $S$ the group $V_n(TS)$ is well defined.

\begin{theorem} \label{theorem: overgroups asynchronous} 
Let $T_{q_0}= \gen{X_n, Q_{T}, \pi_{T}, \lambda_{T}}$ be a  minimal, one-way synchronizing transducer with invers $S_{(\epsilon, q_0)}$. Then $V_n^{{T^{-1}_{q_0}}} = V_n(TS)$.
\end{theorem}

\begin{proof}
First we show that $V_n^{{T^{-1}_{q_0}}} \subseteq V_n(TS)$  

Let $v \in V_n$ and let $B_{p_0} = \gen{X_n, Q, \pi_{B}, \lambda_B}$ be an initial, bi-synchronizing transducer with core equal to the identity transducer representing $v$. Let $\id$ be the core of $B$ (the core of $B$ consists of one state which acts as the identity). Let $j_1 \in \N$ be such that for $\gamma \in X_n^{j_1}$ we have $\lambda_{B}(\gamma, p) = \id$ for all states $p$ of $B_{p_0}$. Let $j_2 \in \N$ be such that for all word $\delta \in X_n^{j_2}$, we have  $|\lambda_{T}(\delta, T_{q_0})|>j_i$, and moreover that $\lambda_{B}(\lambda_{T}(\delta, T_{q_0}),p_0)$ is non-empty.

Therefore given any word $\delta \in X_n^{j_2}$, after processing $\delta$ through $T_{q_0} B_{p_0} S_{(\epsilon, q_0)}$ the active state is of the form $(t, \id, s)$ for appropriate states $t$ of $T$ and $s$ of $S$. Since $\id $ is the single state identity transducer, we may identify $(t, \id, s)$ with the state $(t,s)$ of $TS$. This means that after processing any word from $X_n^{j_2}$ the active state of the transducer $T_{q_0} B_{p_0} S_{(\epsilon, q_0)}$ acts a state of $TS$.

Let $\{\delta_1,\ldots,\delta_{n^j_2}\}=X^{j_2}_n$. Since $T_{q_0} B_{p_0} S_{\epsilon, q_0}$ is a homeomorphism and $X_n^{j_2}$ is a maximal anti-chain for $\CCn$, then the following tuple is a viable combiniation: \[\Big(\big((\delta_1)T_{q_0} B_{p_0} S_{(\epsilon, q_0)}, \ldots, (\delta_{n^{j_2}})T_{q_0} B_{p_0} S_{(\epsilon, q_0)}\big), \big(r_1, \ldots, r_{n^{j_2}}\big)\Big)\] where $r_a$ is the state of $TS$ acting as the active state of $(\delta_a)T_{q_0} B_{p_0} S_{(\epsilon, q_0)}$. It follows that $T_{q_0} B_{p_0} T^{-1}_{q_0} \in V_n(TS)$.

In order to show that $V_n(TS) \subseteq V_n^{T^{-1}_{q_0}}$, it suffices to show that $h^{T_{q_0}} \in V_n$ for all homeomorphisms $h$ which are completions of $TS$. Therefore fix a homeomorphism $h$ a completion of $TS$.

We begin by claiming that there is a fixed $r$ such that given any pair $(s,t)\in Q_S\times Q_T$, word $\alpha \in X_n^{r}$, and word $\rho \in \CCn$, $(\alpha \rho)(ST)_{(s,t)} = \beta \rho$ for some $\beta \in X_n^{\ast}$. Let $\gamma \in X_n^\ast$ be such that $\pi_{S}(\gamma, (\epsilon, q_0)) =s$ and $k$ be the synchronizing level of $T$.  Also let $j\in \N$ such that for all $\delta\in X_n^j$, we have $|\lambda_S(\delta,s)|>k$. We then see \[\pi_{ST}(\gamma\delta, ((\epsilon, q_0,), q_0)) = \pi_{ST}(\delta, (s,t'))= \pi_{ST}(\delta, (s,t))\] for a specific $t'\in Q_T$. This implies that $\pi_{ST}(\delta, (s,t))$ is a state of $T^{-1}_{q_0}T_{q_0}$, the identity homeomorphism, and the claim follows setting $r=j$.

Since $h$ is a completion of  $TS$, there is $j_1 \in \N$ such that for $\alpha \in X_n^{j_1}$ and $\rho \in \CCn$, we have $(\alpha\rho) h = \beta (\rho)(TS)_{(t,s)}$ for some $\beta \in X_n^{*}$ and $(t,s)\in Q_T\times Q_S$.  Let $j_2 \in \N$ be such that for any word $\gamma \in X_n^{j_2}$, we have $|\lambda_{S}(\gamma, (\epsilon,q_0))|  > j_1$.  Now consider $h^{T_{q_0}} = S_{(\epsilon, q_0)} h T_{q_0}$. Let $\gamma \in X_n^{j_2}$, and let $\chi \in \CCn$ be arbitrary. We then see $(\gamma \chi)h^{T_{q_0}} = \delta (\chi)S_{q'}(TS)_{(t,s)} T_{q}$ where $q', s\in Q_S$ and $q, t\in Q_T$. We may re-bracket $S_{q'}(TS)_{(t,s)}T_{q}$ as  $(ST)_{(q', t)}(ST)_{(s,q)}$. It follows by the claim in the previous paragraph that $(ST)_{(q', t)}(ST)_{(s,q)}$ acts as a prefix replacement twice.  Specifically, there is an $r' \in \N$ such that, for any word $\alpha \in X_n^{r'}$ and any $\rho \in \CCn$, $(\alpha \rho)(ST)_{(q', t)}(ST)_{(s,q)} = \delta\rho$ for some $\delta' \in X_n^{*}$. 

We can conclude that for $\gamma\in X_n^{j_2+r'}$ and $\chi\in \CCn$,  $(\gamma\chi)h^{T_{q_0}}=\delta\chi$ for some $\delta\in X_n^*$, so $h^{T_{q_0}} \in V_n$ as required.

\end{proof}

\begin{remark}
If $T$ in the theorem above is a synchronous, one-way synchronizing, invertible transducer then we recover Theorem~\ref{Theorem: ovegroups synchronous}.
\end{remark}

\section{The group \texorpdfstring{$\T{H}(T)$}{Lg} for \texorpdfstring{$T \in \mathfrak{P}_{n}$}{Lg}}\label{Section: H(T)}
In this section, we investigate the groups $\T{H}(T)$ for $T \in \mathfrak{P}_{n}$ further. The aim of this section is to show that whenever $T$ in $\mathfrak{P}_{n}$  satisfies some technical condition, the derived subgroup  of $\T{H}(T)$ is simple.

We begin with some definitions.

\begin{definition}\label{Def:fullgroups}
Let $G$ be a group of homeomorphisms of a topological space $X$. A homeomorphism $h$ of $X$ is said to \emph{locally agree with $G$} if for every point $x \in X$, there is an open neighbourhood $U$ of $X$ and $g \in G$, such that $h\vert_{U} = g\vert_{U}$. The group $G$ is called \emph{full} if every homeomorphism of $X$ which locally agrees with $G$ is an element of $G$.
\end{definition}

\begin{remark}\label{Remark:localhomeosagreeonfiniteclopencover}
Let $G$ be a full group of homeomorphisms of $\CCn$ and $h$ be a homeomorphism of $\CCn$ that locally agrees with $G$. Since $\CCn$ is compact, there is a finite set $\{u_1, u_2, \ldots, u_r\} \subset X_{n}^{*}$  and $g_1, g_2 \ldots, g_r \in G$ such that $h\vert_{[u_i]} =  g_{i} \vert_{[u_i]}$ for all $1 \le i \le r$.
\end{remark}

\begin{definition}\label{Def:flexible}
Let $G$ be a group of homeomorphisms of a topological space $X$. The group $G$ is called \emph{flexible} if for every pair $E_1, E_2$ of proper non-empty clopen subsets of $X$, there is a $g \in G$ such that  $(E_1)g \subset E_2$.
\end{definition}

\begin{remark}\label{Remark:VnTisfullandflexible}
Let $T$ be an invertible, synchronous transducer, then $V_{n}(T)$ is full and flexible. That $V_{n}(T)$ is flexible follows since $V_{n}\le V_{n}(T)$ is flexible. To see that $V_{n}(T)$ is full, we observe that elements of $V_{n}(T)$ are precisely those which locally agree with $\T{H}(T)$. In particular, a  homeomorphism of $\CCn$ which locally agrees with  $V_{n}(T)$, locally agrees with $\T{H}(T)$.
\end{remark}

The following result concerning full and flexible group of homeomorphisms of Cantor space is a restatement of a result of Matui (\cite{HMatui}) which can be found in \cite{JBelkFMatucciJHyde}

\begin{theorem}\label{Thm:fullflexiblegroupshavesimplederivedsubgroups}
Let $G$ be a full, flexible group of homeomorphisms of $\CCn$, then $[G,G]$ is simple.
\end{theorem}

We now discuss when $\T{H}(T)$ is both flexible and full.  In the following lemmas, we demonstrate that for all $T \in \mathfrak{P}_{n}$ such that $T$ has at least one completion in $H(\CCn)$, $\T{H}(T)$ is flexible.  In essence, we argue that given any two prefixes, there exists a  completion of an appropriate transducer that that can exchange those prefixes.

\begin{lemma}\label{lemma:buildingcompletionsfromothers}
Let $T \in \mathfrak{P}_{n}$ and let $h \in H(\CCn)$ be a completion of $T$. Let $g \in V_{n}$, and let $\{ \eta_1, \ldots \eta_{l}\}$ and $\{\nu_1, \ldots, \nu_{l}\}$ be  maximal antichains such that, for $\delta \in \CCn$ and for $1 \le i \le l$, $(\eta_i\delta)g = \nu_i \delta$. Define a map $f$ by $\eta_i \delta \mapsto \nu_i (\delta)h$ then $f$ is also a completion of $T$.
\end{lemma}
\begin{proof}
Let $g \in V_{n}$. There are maximal antichains  $\{ \eta_1, \ldots \eta_{l}\}$ and $\{\nu_1, \ldots, \nu_{l}\}$ of $X_{n}^{*}$ such that, for $\delta \in \CCn$ and $i \in \{1, \ldots,l\}$, the action of $g$ on $\CCn$ is given by $(\eta_{i}\delta)g = \nu_i\delta$. Let $h$ be a completion of $T$. By definition, there is a viable combination $((\rho_1, \ldots, \rho_{r}), (p_1, \ldots, p_r))$ of $T$ where $r \in \N$ and a maximal antichain $(\mu_1, \ldots, \mu_r)$, such that for any $\delta \in \CCn$ and for $1 \le i \le r$, $(\mu_i \delta)h = \rho_i (\delta)T_{p_i}$. For $1 \le i \le l$, set $\vec{\rho}_{i}:= (\nu_i\rho_1, \ldots, \nu_i\rho_{r})$ and $\vec{p}_{i} =(p_1, \ldots, p_r)$. Observe that the set $((\vec{\rho}_1, \ldots, \vec{\rho}_{l}),(\vec{p}_{1}, \ldots, \vec{p}_{l})$ is a viable combination for $T$. Moreover the map $f'$ given by $\eta_i\mu_j\delta \mapsto \nu_i\rho_j (\delta)T_{p_j}$, for $1 \le i \le l$, $1 \le j \le r$ and $\delta \in \CCn$, is in fact equal to the map $f$ given by $\eta_i\gamma \mapsto \nu_i(\gamma)h$, for $1 \le i \le l$ and $\gamma \in \CCn$. We therefore conclude, since $f'$ is a completion of $T$, that $f$ is a completion of $T$.
\end{proof}

\begin{proposition}\label{Prop:groupofcompletionsisfelxible}
Let $T \in \mathfrak{P}_{n}$, and suppose that $T$ has one completion in $H(\CCn)$, then $\T{H}(T)$ is flexible.
\end{proposition}
\begin{proof} 
Let $E_1$ and $E_2$. Let $g \in V_{n}$ be an element  such that $(E_1)g \subset E_2$. Let $\{\eta_1, \ldots, \eta_r \}$, $\{\nu_1, \ldots, \nu_r \}$ be complete antichains such that $(\eta_i\delta)g = \nu_i \delta$ for $1 \le i \le r$ and $\delta \in \CCn$. We may assume since $(E_1)g \subset E_2$, that there is some $1 \le j \le r$ and $0 \le k \le r-j$ such that $[\eta_j], [\eta_{j+1}], \ldots,  [\eta_{j+k}]$ is an open cover of $E_1$ and for $j \le l \le j+k$ $[\nu_l] \subset E_2$. Let $h$ be a completion of $T$. Then the map $f$ given by $\eta_i \delta \mapsto \nu_i (\delta)h$ is also a completion of $T$, by Lemma~\ref{lemma:buildingcompletionsfromothers}, and $(E_1)f \subset E_2$. 
\end{proof}

Note that the group $\T{H}(T)$ is not full in general. Let $T$ be the single state transducer over the alphabet $X_{2}$ that transforms the input $1$ to $0$ and $0$ to $1$. Observe that an element $h$ of $\T{H}(T)$ is either in $V$, or there is a clopen  cover $\{[\nu]_{i} \mid \ 1 \le i \le l\}$ such that, for $\rho \in \mathfrak{C}_{2}$, $(\nu_i\rho)h = \delta (\rho)T$.  Therefore, the homeomorphism $g$ of $\C{C}{2}$ mapping $0\delta \to 0 \delta$ and $1\delta \to 1 (\delta)T$ for $\delta \in \C{C}{2}$ is not an element of $\T{H}(T)$. Elements of $\T{H}(T)$ are either elements of $V_{2}$ or act as $T$ on a clopen partition of $\C{C}{2}$.

By Remark~\ref{Remark:VnTisfullandflexible}, $V_{n}(T)$ is the group of all homeomorphisms of $\CCn$ that locally agree with $\T{H}(T)$ for synchronous $T$. As we have seen, $V_{n}(T)$ is not equal to $\T{H}(T)$ is general, e.g.~in the previous paragraph $g\in V_n(T)$ but $g\not\in\T{H}(T)$. However, whenever $T$ is synchronous and has a state inducing the identity homeomorphism, then $V_{n}(T)$ is in fact equal to $\T{H}(T)$. To consider the asynchronous case, we make the following definition:

\begin{definition}\label{Def:fullgroupforasynchronoustransducer}
Let $T  \in \mathfrak{P}_{n}$ and suppose that $T$ has one completion in $H(\CCn)$,  then set $\T{F}(T)$ to be the group of all homeomorphisms which locally agree with $\T{H}(T)$.
\end{definition}

As a corollary of Theorem~\ref{Thm:fullflexiblegroupshavesimplederivedsubgroups} we have:

\begin{corollary}
Let $T \in \mathfrak{P}_{n}$ and suppose that $T$ has a completion in $H(\CCn)$. The group $[\T{F}(T), \T{F}(T)]$ is simple.
\end{corollary}

For the remainder of the section we investigate simple criteria which imply the equality $\T{F}(T) = \T{H}(T)$.

The following construction allows us to give a more explicit description of the group $\T{H}(T)$.  Our construction generalises the inverse to partially invertible transducers, as the name suggests.

Let $T=\langle X_n,Q_T, \pi_T,\lambda_T\rangle$ be a partially invertible transducer. For each state $q$ of $T$, since $\im(T_q)$ is clopen, we may find  clopen subsets  $[\eta_{1,q}] \ldots, [\eta_{r_q,q}]$ of $\CCn$  such that $\bigcup_{1 \le i \le r_q} [\eta_{i, q}]  = \im(T_q)$.  We may further assume that $\{[\eta_{1,q}] \ldots, [\eta_{r_q,q}] \}$ is a minimal set of basic open sets with this property. This means that if $q$ a homeomorphism state of $T$, then  $r_q = 1$, and $\eta_{1,q} = \epsilon$. 

We recall to remind the reader, that for $q\in Q_T$ and $\alpha\in X_n^*$, that $\T{L}_{q}(\alpha)$ is the longest common prefix of the words in $T_q([\alpha])$.  Then for $q$ in $Q_{T}$ and $1 \le i \le r_{q}$,  let $\nu_{i,q} = \T{L}_{q}(\eta_{i,q})$, $w_{i,q}= (\eta - \lambda_{T}(\nu_{i,q}, q))$ and $p_{i,q} = \pi_{T}(\nu_{i,q}, q)$. Notice that if $q$ is a homeomorphism state of $T$, then $w_{1,q} = \epsilon$ and $p_{1,q} = q$. Set $Q'_{0} = \{(w_{i,q}, p_{i,q}) \mid  q \in Q_{T}, 1 \le i \le r_{q} \}$. We observe that for $q \in Q$ and $1 \le i \le r_q$, $[w_{i,q}] \subset \im(T_{p_{i,q}})$ and $\T{L}_{p}(w_{i,q}) = \epsilon$. Inductively, for $k \in \N\backslash\{0\}$, set 
$$Q'_{k} = \{(wx - \lambda_{T}(\T{L}_{p}(wx),p), \pi_{T}(\T{L}_{p}(wx),p))\mid x \in \xn \mbox{ and }(w,p) \in Q'_{k-1}  \} \cup Q'_{k-1}.$$

For all $k$, a simple induction argument shows that if $(w,p) \in Q'_{k}$, then $\T{L}_{p}(w) = \epsilon$ and $[w]  \subset \im(T_{p})$. Since  for all states $p \in Q_{T}$, $T_{p}$ is injective and has clopen image, it follows that there are only finitely many words $w \in \xns$ such that $[w] \subset \im(T_{p})$, and $\T{L}_{p}(w) = \epsilon$. Therefore, there is a $j \in \N$ such that  $Q'_{j}= Q'_{j+1}$. Set $Q_{T'}:= Q'_{j}$. 

We form a transducer $T' = \gen{ \xn, Q_{T'}, \pi_{T'}, \lambda_{T'}}$ with $\pi_{T'}$ and $\lambda_{T'}$ obeying the following rules. For $x \in \xn$ and $(w,p) \in Q_{T'}$ we have, $\tr{T'}(x,(w,p)) = (wx - \lambda_{T'}(\T{L}_{p}(wx),p), \pi_{T}(\T{L}_{p}(wx),p))$ and $\out{T'}(x, (w,p)) = \T{L}_{p}(wx)$. 

We make the following observations:
\begin{enumerate}
\item For $\delta \in \CCn$ and $(w,p) \in Q_{T'}$,  by transfinite induction we have, $\out{T'}(\delta, (w,p)) = (w\delta)T^{-1}_{p}$. Thus, for $(w,p) \in Q_{T'}$, $T'_{(w,p)}$ is injective.
\item  Since each state of $T$ is injective and has clopen image, and since for each $(w,p) \in Q_{T'}$, $[w] \subset \im(T_{p})$, then $\im(T'_{(w,p)})$ also has clopen image.
\end{enumerate}

\begin{definition}\label{Def:partialinverse}
Let $T \in \mathfrak{P}_{n}$, let $T'$ be the transducer constructed as above. We call $T'$ the \emph{partial inverse of $T$}.
\end{definition}

We have the following lemma.

\begin{lemma}\label{Lemma:inverseofcompletioniscompletionofpartialinverse}
Let $T \in \mathfrak{P}_{n}$. Let $h$ be any completion of $T$ in $H(\CCn)$, then $h^{-1}$ is a completion of $T'$.
\end{lemma}
\begin{proof}
Let $T \in \mathfrak{P}_{n}$ and let $h$ be any completion of $T$ in $H(\CCn)$. Since there is a $k \in \N$ such that for any $\delta \in \CCn$, after processing a prefix the length $k$ prefix of $\delta$ $h$ acts on the remaining suffix as a state of $T$,  we observe that there is an invertible, initial transducer $A_{q_0} = \gen{ \xn, Q_{A}, \tr{A}, \out{A}}$ such that $Q_{T} \subset Q_{A}$ and $A_{q_0} : \CCn \to \CCn$ coincides with $h$. Moreover, for any word $\Gamma \in \xn^{k}$ and any state $q \in Q_{A}$, we have $\tr{A}(\Gamma,q) \in Q_{T}$. We now apply the inversion algorithm to $A_{q_0}$. Let $B_{\epsilon, q_0}$ be the initial transducer that is the inverse of $A_{q_0}$

Let $\Gamma \in \xn^{k}$ and let $p \in Q_{T}$ be such that  $\tr{A}(\Gamma, q_0) = p$ and $\out{A}(\Gamma, q_0) = \Xi$. Let $\eta \subset \im(T_{p})$ and $\nu =\T{L}_{q_0}(\Xi\eta)$. Since $h$ is injective, we must have that $\nu = \Gamma \T{L}_{p}(\eta)$. Therefore $(\Xi\eta - \out{A}(\nu,q_0), \tr{A}(\nu, q_0)) = (\eta - \out{T}(\nu, q_0), \tr{T}(\T{L}_{p}(\eta), p))$. Hence, in $B$ we have $\out{B}(\Xi\eta, (\epsilon, q_0)) = (\eta - \out{T}(\nu, q_0), \tr{T}(\T{L}_{p}(\eta), p))$. Therefore there is a $j \in \N$ such that after processing any word of length $j$, $B_{(\epsilon, q_0)}$ acts as a state of $T'$. Now since $B_{(\epsilon, q_0)}$ coincides with $h^{-1}$ in its action on $\CCn$, we conclude that $h^{-1}$ is a completion of $T'$. 
\end{proof}

\begin{remark}\label{Remark:relationbetweenT'andSwhenTisinvertible}
Let $T \in \mathfrak{P}_{n}$ and suppose there is a state $q_0 \in Q_{T}$ such that $T_{q_0}$ is an accessible and invertible transducer. Let $S_{(\epsilon, q_0)}$ be the inverse of $T$, then, by Lemma~\ref{Lemma:inverseofcompletioniscompletionofpartialinverse}, every state of $T'$ is a state of $S_{(\epsilon, q_0)}$ and vice versa since $(\epsilon, q_0)$ is a state of $T'$. Note that $TS$ is the full transducer product of $T$ and $S$, and even though $T_{q_0}$ and $S_{(\epsilon, q_0)}$ are invertible, $TS$ is non-initial and does not necessarily have a well defined inverse, however its partial inverse is defined. Observe that in the case that $T_{q_0}$ is asynchronous, then $TS$ has injective but not surjective states.
\end{remark}

For $T \in \mathfrak{P}_{n}$, let $\Pi(T)$ be the set of all finite products of $T$ and $T'$. By Lemma~\ref{Lemma:inverseofcompletioniscompletionofpartialinverse},  it follows that an element of $\T{F}(T)$ corresponds to a completion arising from a viable combination of a transducer equal to the disjoint union of finitely many elements of $\Pi(T)$. This is because for each element $h$ of $\T{H}(T)$, there is a $j \in \N$ and a transducer $U \in \Pi(T)$ such that given $\delta \in \CCn$, after replacing the length $j$ prefix of $\delta$, $h$ acts as a state of $U$. Thus it follows, in the case where $T$ is an invertible, synchronous transducer, that $V_{n}(T) = \T{F}(T)$. If in addition to being synchronous and invertible, $T$ possesses a state inducing the identity homeomorphism, then $V_{n}(T) = \T{H}(T) = \T{F}(T)$.  

We now introduce a criterion for which, given $T \in \mathfrak{P}_{n}$, synchronous or asynchronous, with a completion in $H(\CCn)$, $\T{H}(T) = \T{F}(T)$.

\begin{definition}\label{Def:contractingtransducers}
Let $T \in \mathfrak{P}_{n}$, we say that $T$ is contracting, if, given a transducer $U \in \Pi(T)$, for any state $q$ of $U$, $U_{q}$ is $\omega$-equivalent to an initial transducer $V_{p}$ for which there is a $j \in N$ such that after reading a word of length $j$, $V_{p}$ acts as a state of $T$.
\end{definition}

Contracting transducers have been well studied in the synchronous case and include the transducers generating the Grigorchuk group and the Gupta-Sidki group and some examples arising from iterated monodromy groups (\cite{BartholdiGrigorchukNekrashevych}). More generally, we show that for $T_{q_0}$ an  invertible (synchronous or asynchronous) and synchronizing transducer, with $S_{(\epsilon, p_0)}$ its inverse, then  $TS$ and $ST$ are both contracting.

First we require the following lemma.

\begin{lemma}\label{Lemma:statesof(TS)'areequaltostatesofTS}
Let $T_{q_0}$ be a  transducer without states of incomplete response, and inverse $S_{(\epsilon, p_0)}$. Let $TS$ be the product of $T$ and $S$ and $(TS)'$ be the partial inverse of $TS$. There is a $k \in \N$ such that for any state $p$ of $(TS)'$, and any word $\Gamma \in \xn^{k}$, $\tr{(TS)'}(\Gamma, p)$ is $\omega$-equivalent to a state of $TS$. 
\end{lemma}
\begin{proof}
 Let $(w,s)$ be a state of $S_{(\epsilon, p_0)}$ and $t$ a  state of $T_{q_0}$. Let $v \in \xns$ be such that $[v] \subset  \im(TS_{(t,(w,s))})$ and $\T{L}_{(t,(w,s))}(v) = \epsilon$. Now observe that $[v]$ must be a subset of $\im(S_{(w,s)})$. Moreover  $\T{L}_{(t,(w,s))}(v)$ is the greatest common prefix of the set $\{ \rho \in \CCn \mid v \le (\rho)TS_{(t, (w,s))}  \}$. Let $\mathscr{S}:= \{ \rho \in \CCn \mid (\rho)S_{(w,s)} \in [v]  \}$, then $\T{L}_{(t,(w,s))}(v)$ is the greatest common prefix of the set $\{ \rho \in \CCn \mid v \le (\rho)T_{t} \in \mathscr{S}\}$. Let $\mathscr{T}:= \{ \rho \in \CCn \mid v \le (\rho)T_{t} \in \mathscr{S}\}$.

We observe that since $[v] \subset \im(S_{(w,s)})$, it must be the case that $\out{T}(v, s) = w \nu$ for some $\nu \in \xns$. This is because for all $x$, there is a word $u \in \xns$ such that $\T{L}_{s}(wu)$ has prefix $vx$ and $T$ is assumed to have no states of incomplete response. Let $\tr{T}(v,s) = r$, we must have $\mathscr{S}:= \{\nu\im(r) \}$ since all states of $S$ induce injective maps. Thus $\mathscr{T} = \{ \rho \in \CCn \mid (\rho)T_{t} \in \nu \im(r)\}$. Now since $\T{L}_{(t,(w,s))}(v) = \epsilon$, the greatest  common prefix of $\mathscr{T}$ is $\epsilon$ and we must have  $\T{L}_{t}(\nu) = \epsilon$.

We now show that the state $(v, (t,(w,s)))$ of $(TS)'$ is $\omega$-equivalent to the state $(r,(\nu,t))$. Note that $(\nu, t)$ might not satisfy $[\nu] \subset \im(T)$, and so might not be a state of $S_{(\epsilon, q_0)}$. However we may extend the transition and output functions of $S$ to the state $(\nu,t)$. We do this by defining, as in Proposition~\ref{Proposition:inversionalgorithm}, $\lambda_{S}(x, (\nu,t)) = \T{L}_{t}(\nu x)$ and $\pi_{S}(x, (\nu,t)) = (\nu x - \lambda_{T}(\T{L}_{t}(\nu x),t), \pi_{T}(\T{L}_{t}(\nu x),t))$ for $x \in \xn$, since $\T{L}_{t}(\nu) = \epsilon$. 

Let $\gamma \in \xnp$ we compute $\out{(TS)'}(\gamma, (v, (t,(w,s)))) = \T{L}_{(t,(w,s))}(v\gamma)$. By the arguments above, we see that $\T{L}_{(t,(w,s))}(v\gamma)$ is the greatest common prefix of the set $\{\rho \in \CCn \mid (\rho)T_{t} \in \nu\mu\im{r'}\}$, where $w\nu\mu = \out{T}(v\gamma,s)$ and $r' = \tr{T}(v\gamma,s)$. Let $\phi$ be the greatest common prefix of the set $\{\rho \in \CCn \mid (\rho)T_{t} \in \nu\mu\im{r'}\}$. As $\lambda_{T}(\gamma,r) = \mu$, we see that $\T{L}_{t}(\nu\mu)$ is a prefix of $\phi$. Moreover since $\nu\im(r)$ is a subset of $\im(r)$, and $T_{t}$ is injective, then as $\gamma \in \xnp$ increases in length we see that $\T{L}_{t}(\nu\out{T}(\gamma, r))$ also increases in length. Thus, by transfinite induction, we conclude that for $\rho \in \CCn$, $(\rho)(TS)'_{(v,(t,(w,s)))} = (\rho)TS_{(r,(\nu,t))}$.

In the case that $(\nu,t)$ is not accessible from $(\epsilon, q_0)$, observe that since $\im(r)$ is clopen, then $\nu \im(r)$ is also clopen and so there is a length $j \in \N$ such that for all inputs $\gamma \in \xn^{j}$, 
$[\nu\out{T}(\gamma,r)] \subset \im(t)$. Therefore $(\nu\out{T}(\gamma,r) - \out{T}(\T{L}_{t}(\nu\out{T}(\gamma,r)),t), \tr{T}(\T{L}_{t}(\nu\out{T}(\gamma,r)),t))$ is a state of 
$T'$ and so a state of $S_{(\epsilon, q_0)}$ by Remark~\ref{Remark:relationbetweenT'andSwhenTisinvertible}.
\end{proof}

\begin{proposition}\label{Proposition:TSandSTcontractingifTissycnh}
Let $T_{q_0}$ be a  synchronizing transducer without states of incomplete response and inverse $S_{(\epsilon, p_0)}$, then  $TS$ and $ST$ are both contracting.
\end{proposition} 
\begin{proof}
First observe that given an state $s \in Q_{S}$ and $t \in Q_{T}$, then $ST_{(s,t)}$ is a prefix replacement map inducing a continuous function from $\CCn \to \CCn$. The proof follows precisely as in the proof of Proposition~\ref{pro:conjproducessubgroups}. Let $k$ be the synchronizing level of $T_{q_0}$ and let $j \in \N$ be such that for all words $\Gamma \in \xn^{j}$, $|\out{S}(\Gamma,s)| \ge k$. Let $ \gamma \in \xns$ be such that $\tr{S}(\gamma, p_0) = s$, and let $\Gamma \in \xn^{j}$ be arbitrary, then observe that $\tr{ST}(\gamma\Gamma,(p_0, q_0)) = \tr{ST}(\Gamma, (s,t))$. Let $(s',t':=\tr{ST}(\Gamma, (s,t))$, we have that $(s',t')$ is a state of  $ST_{(p_0, q_0)}$. Now since $TS_{(p_0, q_0)}$ is $\omega$-equivalent to the identity transducer, it follows that if $\Lambda$ is the greatest common prefix of the set $\{ (\delta)ST_{(s',t')} \mid \delta \in \CCn \}$, then for all $\Delta \in \xnp$,$\out{ST}(\Delta, (s',t')) - \Lambda = \Delta$. Thus we see that $ST_{(s',t')}$ is a prefix replacement map, that prepends the prefix $\Lambda$ to inputs.

Now consider a state, $p$, of a transducer in $\Pi(TS)$. It corresponds to a product $p_1p_2\ldots p_l$ of states of $(TS)$ and $(TS)'$. Now by Lemma~\ref{Lemma:statesof(TS)'areequaltostatesofTS}, there is a $j$ such that after processing any word of length $j$ from $p$, the resulting states is $\omega$-equivalent to a product $q_1q_2\ldots q_l$ of states of $(TS)$. As $q_1q_2\ldots q_l$ is a product of states of $(TS)$, it is $\omega$-equivalent to a product  $t_1(s_1t_2)\ldots(s_{l-1}t_{l})s_{l}$ where the $t_i$'s are states of $T$ and $s_i$'s are states of $S$. By arguments in the preceding paragraphs $(s_1t_2)\ldots(s_{l-1}t_{l})$ is a prefix replacement and so is $\omega$-equivalent to an initial, synchronizing transducer with trivial core. Hence, the product $t_1(s_1t_2)\ldots(s_{l-1}t_{l})s_{l}$ is $\omega$-equivalent to an initial transducer for which there is a fixed $j \in N$ such that after processing a word of length $j$, it acts as a state of $T$.
\end{proof}

We have the following result:

\begin{theorem}\label{Theorem:ifTiscontactingthenhaveequality}
Let $T \in \mathfrak{P}_{n}$ be a contracting transducer with at least one completion in $H(\CCn)$, then $\T{H}(T) = \T{F}(T)$.
\end{theorem}
\begin{proof}
It is clear that $\T{H}(T) \subset \T{F}(T)$. We show that $\T{F}(T) \subset \T{H}(T)$. By definition, for an element $h$ of $\T{F}(T)$, there is a $j \in N$ such that after processing an initial finite prefix it acts as a state of an element of $\Pi(T)$. Since $T$ is contracting, this means that there is a $k \in N$ such that after reading a word of length $k$, $h$ acts as an initial transducer $G_{p}$ which eventually process all inputs as a state of $T$.  In total, we see there is an $l \in N$ such that after processing any  word of length $l$, $h$ acts as a state of $T$ and so is a completion of $T$. 
\end{proof}

\begin{remark}
Notice that when $T$ is synchronous and invertible, we have $\T{H}(T) = \T{F}(T)$ whenever $T$ contains a state inducing the identity homeomorphism. Clearly the contracting condition in Theorem~\ref{Theorem:ifTiscontactingthenhaveequality} is quite strong, and so it is natural to ask if it can be replaced with a weaker condition. For instance does the condition that the identity homeomorphism is a completion of $T$ imply that $\T{H}(T) = \T{F}(T)$?
\end{remark}
\section{Open Problems}\label{Section: Questions}

\begin{enumerate}[label = {\bf{Q}\arabic*.}] 
\item   Is it true that $V_{\sim_A}\cong V_n$ if and only if there is a initial, invertible, synchronizing transducer $T_{q_0} = \gen{X_n, Q, \pi, \lambda}$ with inverse$S_{(\epsilon, q_0)} = \gen{X_n, Q, \pi, \lambda}$ where $A= \gen{X_n, Q', \pi', q_0}$? More generally, when is $V_{\sim_A}\cong V_{\sim_B}$?
\end{enumerate}

If $V_{\sim_A}$ acts in a locally dense fashion, then Rubin's Theorem \cite{Rubin} implies that isomorphisms between $V_{\sim_A}$ and $V_n$ must result from conjugation.  The question then becomes a matter of determining under what conditions $V_{\sim_A}$ acts in a locally dense fashion on $\CCn$.

\begin{enumerate}[label = {\bf{Q}\arabic*.}] 
  \setcounter{enumi}{1}
\item When does an automaton $A$ arise as the automaton obtained from the inverse of a synchronizing transducer when we ignore outputs (as in Theorem~\ref{thm:subgroups_asynchronous_case})?  
\end{enumerate}

It is easy to settle this question if we insist that $A$ must arise from the inverse of a \emph{synchronous} synchronizing transducer. This is because inversion does not change the structure of the underlying graph of synchronous transducer (viewing the automaton as a graph) and the underlying graph of a synchronous, synchronizing transducer has very distinct features. However, if we allow $A$ to arise from the inverse of a synchronizing, asynchronous transducer, then it is possible that in inverting we drastically alter the underlying graph structure.
\begin{enumerate}[label = {\bf{Q}\arabic*.}]
  \setcounter{enumi}{2} 
\item  What are the group theoretic properties of $V_{\sim}$ and $V_{\sim_A}$? Are these groups finitely generated or finitely presented?
\end{enumerate}
\begin{enumerate}[label = {\bf{Q}\arabic*.}]
  \setcounter{enumi}{3} 
\item What are the group theoretic properties of $V_{n}(T)\backslash \T{F}(T)$ for $T \in \mathfrak{P}_{n}$? Are they finitely generated? Are they finitely presented? Are they simple or almost simple?
\end{enumerate}

Recent work of Bleak and Hyde show that certain groups acting on Cantor space, including the Nekrashevych groups $V_n(G)$ for $G$ self-similar, are 2-generated. In particular, this shows $V_n(T)$ is 2-generated when $T$ is synchronous. This could be generalisable to the asynchronous case and the subgroups $V_{\sim_A}$. 

\begin{enumerate}[label = {\bf{Q}\arabic*.}]
  \setcounter{enumi}{4} 
\item For $T \in \mathfrak{P}_{n}$ when is $V_{n}(T)\cong V_{n}$?
\end{enumerate}

Theorem~\ref{theorem: overgroups asynchronous} demonstrates that when the transducer $R \in \mathfrak{P}_{n}$ can be expressed as the product $TS$ where $T$ is a minimal, finite, one-way synchronizing transducer and $S$ is its inverse then $V_{n}(TS)$ is an overgroup of $V_n$ isomorphic to $V_{n}$.  This begs the question as to when this is possible and if this is the only way an isomorphism may arise.

\bibliographystyle{amsplain}
\bibliography{references}
\vspace{1.5cm}
\begin{tabularx}{\textwidth}{XX}
   Casey Donoven & Feyishayo Olukoya \\
   Department of Mathematical Science & School Of Mathematics and Statistics\\
Binghamton University & University of St. Andrews \\
PO Box 6000 & St Andrews \\
Binghamton, New York 13902-6000 & Fife KY16 9SS \\
   USA & Scotland \\
   \texttt{cdonoven@binghamton.edu} & \texttt{fo55@st-andrews.ac.uk}
  
\end{tabularx}
\end{document}